\documentclass{article}
\usepackage[utf8]{inputenc}
\usepackage[english]{babel}
\usepackage[]{amsthm} %lets us use \begin{proof}
\usepackage[]{amssymb} %gives us the character\varnothing
\usepackage[margin=1.0in]{geometry}
\usepackage{amsmath}
\usepackage{tikz-cd}
\usepackage{mathrsfs}
\usepackage{graphicx}
\usepackage{listings}
\graphicspath{ {./images/} }
\usepackage{hyperref}

\newtheorem{theorem}{Theorem}

\newtheorem{prop}{Proposition}[section]
\newtheorem{lemma}[prop]{Lemma}

\newtheorem{definition}[prop]{Definition}
\newtheorem{remark}[prop]{Remark}

\newcommand{\CC}{\mathbb{C}}
\newcommand{\RR}{\mathbb{R}}
\newcommand{\HH}{\mathbb{H}}

\newcommand{\ZZ}{\mathbb{Z}}
\newcommand{\NN}{\mathbb{N}}
\newcommand{\PP}{\mathbb{P}}

\newcommand{\del}{\partial}
\newcommand{\delbar}{\bar{\partial}}

\title{Limits of Convex Projective Surfaces and Finsler Metrics}
\author{Charlie Reid}

\date\today

\begin{document}
\maketitle

\begin{abstract}
We show that for certain sequences escaping to infinity in the $\operatorname{SL}_3\RR$ Hitchin component, growth rates of trace functions are described by natural Finsler metrics. More specifically, as the Labourie-Loftin cubic differential gets big, logarithms of trace functions are approximated by lengths in a Finsler metric which has triangular unit balls and is defined directly in terms of the cubic differential. This is equivalent to a conjecture of Loftin from 2006 \cite{loftin_flat_2007} which has recently been proven by Loftin, Tambourelli, and Wolf \cite{loftin_limits_2022}, though phrasing the result in terms of Finsler metrics is new and leads to stronger results with simpler proofs. From our perspective, the result is a corollary of a more local theorem which may have other applications. The key ingredient of the proof is another asymmetric Finsler metric, defined on any convex projective surface, recently defined by Danciger and Stecker, in which lengths of loops are logarithms of eigenvalues. We imitate work of Nie \cite{nie_limit_2022} to show that, as the cubic differential gets big, Danciger and Stecker's metric converges to our Finsler metric with triangular unit balls. While \cite{loftin_limits_2022} addresses cubic differential rays, our methods also address sequences of representations which are asymptotic to cubic differential rays, giving us more insight into natural compactifications of the moduli space of convex projective surfaces.
\end{abstract}

\section{Introduction}
A projective structure on a closed manifold $M$ is an atlas of charts valued in $\RR \PP^n$ with transition functions in $\operatorname{PGL}_{n+1}\RR$. Such a structure gives rise to a developing map $\tilde{M}\to \RR \PP^n$, and a holonomy representation $\pi_1(M)\to \operatorname{PGL}_{n+1}\RR$. A projective structure is called convex if the developing map is a homeomorphism onto a properly convex domain in $\RR \PP^n$. Classification of convex projective manifolds in general is a largely open subject, but the case of surfaces is understood: the space of convex projective structures on an oriented closed surface of genus $g$, which we will denote $\operatorname{Conv}(S)$, is a ball of dimension $16g-16$ \cite{choi_convex_1993}. Taking holonomy representations identifies $\operatorname{Conv}(S)$ with the Hitchin component of the space of representations $\operatorname{Rep}(\pi_1 (S),\operatorname{SL}_3\RR)$. This is analogous to the fact that the space $\operatorname{Teich}(S)$ of hyperbolic structures on $S$, is a ball of dimension $6g-6$, and is in bijection with a component of $\operatorname{Rep}(\pi_1 (S),\operatorname{PSL}_2\RR)$. In fact the symmetric square $\operatorname{PSL}_2\RR\to \operatorname{SL}_3\RR$ embeds $\operatorname{Teich}(S)$ into $\operatorname{Conv}(S)$ as a submanifold, which we call the Fuchsian locus.

This paper is largely motivated by the aspiration to extend Thurston's compactification of $\operatorname{Teich}(S)$ by measured foliations to a compactification of $\operatorname{Conv}(S)$. Features of Thurston's compactification which are desiderata for the convex projective case include the following.
\begin{itemize}
\item Thurston's compactification is homeomorphic to a closed ball.
\item A boundary point records ratios of growth rates of lengths of closed curves.
\item The boundary points parametrize some kind of geometric structures on $S$.
\end{itemize}
In this paper we show that certain sequences in $\operatorname{Conv}(S)$ converge to the following class of Finsler metrics in a natural way. 

\begin{definition}
Let $\mu$ be a cubic differential on a Riemann surface $C$; that is, a holomorphic section of $(T^*C)^{\otimes 3}$. We define $F^\Delta_\mu$ to be the maximum of twice the real parts of the cube roots of $\mu$. 
\[F^\Delta_\mu(v) := \max_{\{\alpha\in T_x^*C:\alpha^3=\mu_x\}} 2 Re(\alpha(v))\]
Here $x\in C$ is a point, and $v\in T_x C$ is a tangent vector.
\end{definition}

We conjecture that these Finsler metrics, considered up to scaling, comprise an open dense subset of a compactification which extends Thurston's compactification of Teichm\"uller space, and satisfies the above three desiterata. The sequences we consider are ``orthogonal" to $\operatorname{Teich}(S)$ in the following sense.

In the early 2000's Labourie \cite{labourie_flat_2006} and Loftin \cite{loftin_affine_2001} found a beautiful, but non-explicit parametrization of the space of convex projective structures on a closed surface $S$ by pairs $(J,\mu)$ where $J$ is a complex structure, and $\mu$ is a holomorphic cubic differential. The Fuchsian locus corresponds to $\mu=0$. We will consider sequences $(J_i,\mu_i)$ where $J_i$ converges, and $\mu_i$ diverges.

 A natural measurement one can make on a convex projective structure is the logarithm of the top eigenvalue of the action of a group element $\gamma\in \pi_1(S)$.
 \[\log(\lambda_1(\rho(\gamma)))\]
 Here, $\rho:\pi_1(S)\to \operatorname{SL}_3\RR$ is the holonomy representation. If $\rho$ is Fuchsian, then this is the hyperbolic length of the geodesic representing the conjugacy class of $\gamma$, so we think of it as a notion of geodesic length for closed curves in convex projective manifolds. We will call $\log(\lambda_1(\rho(\gamma)))$ the asymmetric length, because it is a function of oriented loops, and to distinguish it from the more commonly used Hilbert length: 
$\log(\lambda_1(\rho(\gamma))/\lambda_3(\rho(\gamma)))$.
Hilbert length is the symmetrization of the asymmetric length.

\begin{theorem}
\label{evlimit}
Let $\mu_i$ be a sequence of cubic differentials on a smooth oriented surface $S$ of genus at least $2$, each holomorphic with respect to a complex structure $J_i$, such that $a_i^3\mu_i$ converges uniformally to $\mu$, for some sequence of positive real numbers $a_i$ tending to $0$. (It follows that $J_i$ converge to $J$.)  Let $\gamma\in \pi_1(S)$. Let $F^{\Delta}_\mu(\gamma)$ denote the infimal length of loops representing $\gamma$ in the Finsler metric $F_\mu^{\Delta}$.
\[\lim_{i \to \infty} a_i \log(\lambda_1(\rho(\gamma))) = F^{\Delta}_\mu(\gamma)\]
\end{theorem}
Part of the appeal of theorem \ref{evlimit} is that the left hand side is complicated, necessitating both solution of a PDE and an ODE to compute directly from $\mu_i$, while the right hand side only involves integrals of cube roots of $\mu$. The cubic differential $\mu$ is equivalent to a $1/3$-translation structure on $S$ in which $F^\Delta_\mu$ geodesics can always be straightened to be concatenations of straight line segments which have angle at least $\pi$ on either side at zeros. To compute $F^{\Delta}_\mu(\gamma)$, it suffices to find such a geodesic representative and add up the $F^\Delta_\mu$ lengths of the constituent segments. Theorem \ref{evlimit} was proved in \cite{loftin_limits_2022} for the case of cubic differential rays $\mu_i = \mu/a_i^3$. We are not sure if one can directly deduce Theorem \ref{evlimit} from \cite{loftin_limits_2022}, but in any case, theorem \ref{evlimit} is a useful improvement.
 
 Our main new result is a more local version of this theorem. Danciger and Stecker have recently discovered an asymmetric Finsler metric $F^{DS}$, defined on any closed convex projective manifold, whose length function is the asymmetric length function.
 \[F^{DS}(\gamma) = \log(\lambda_1(\rho(\gamma)))\]
  We call it the domain shape metric because its unit balls are all projectively equivalent to the developing image of the convex projective structure. Our main theorem says that when $\mu$ is large, the domain shape metric for the projective structure on $S$ specified by $(J,\mu)$, looks roughly like a much simpler Finsler metric $F^\Delta_\mu$.

\begin{theorem}
\label{metriclimit}
If $S, \mu_i$, $a_i$ and $\mu$ are as in theorem \ref{evlimit}, then $a_i F^{DS}_{\mu_i}$ converges uniformally to $F^{\Delta}_{\mu}$ on any compact set in the complement of the zeros of $\mu$.
\end{theorem}
Pointwise convergence would follow from \cite{nie_limit_2022} which describes the Gromov-Hausdorff limit of a sequence of pointed convex domains coming from a sequence of pointed Riemann surfaces with cubic differentials tending to infinity, but we will need uniformity to deduce theorem $\ref{evlimit}$ from theorem $\ref{metriclimit}$. To get uniform convergence away from zeros, we will retrace the steps of \cite{nie_limit_2022}, with a slightly different setup, and make sure things work uniformly. In section 5 we show that uniform convergence away from zeros is sufficient for deducing convergence of length functions.

\subsection{What theorem \ref{evlimit} says about compactification}
Let $\mathcal{L}(S)$ denote the set of homotopy classes of closed loops in $S$. An attractive way to define a compactification of $\operatorname{Conv}(S)$, which goes by various names, including Morgan-Shalen compactification, tropical compactification, and spectral radius compactification, is to embed it into $\PP(\RR^{\mathcal{L}(S)})$ via taking projectivized marked asymmetric length spectrum, then take the closure. The map $\operatorname{Conv}(S)\to \PP(\RR^{\mathcal{L}(S)})$ is an embedding, but the topology of the closure is still unknown. Part of the difficulty is that we still don't have a good understanding of what geometric structures on $S$ the tropical boundary points might be parametrizing. Theorem \ref{evlimit} takes a step towards answering this last question.

The Labourie-Loftin parametrization identifies $\operatorname{Conv}(S)$ with a vector bundle over Teichmuller space $Q(S)$ which we can compactify in the fiber directions by adding in a point for each $\RR_+$ orbit. Let $\bar{Q}(S)$ denote this radial partial compactification. Theorem \ref{evlimit} implies that the projectivized marked length spectrum map $\operatorname{Conv}(S)\to \PP(\RR^{\mathcal{L}(S)})$ extends continuously to $\bar{Q}(S)$, and the boundary map is given by the projectivized length spectrum of $F^\Delta_{\mu}$.
\begin{center}
\includegraphics[width=8cm]{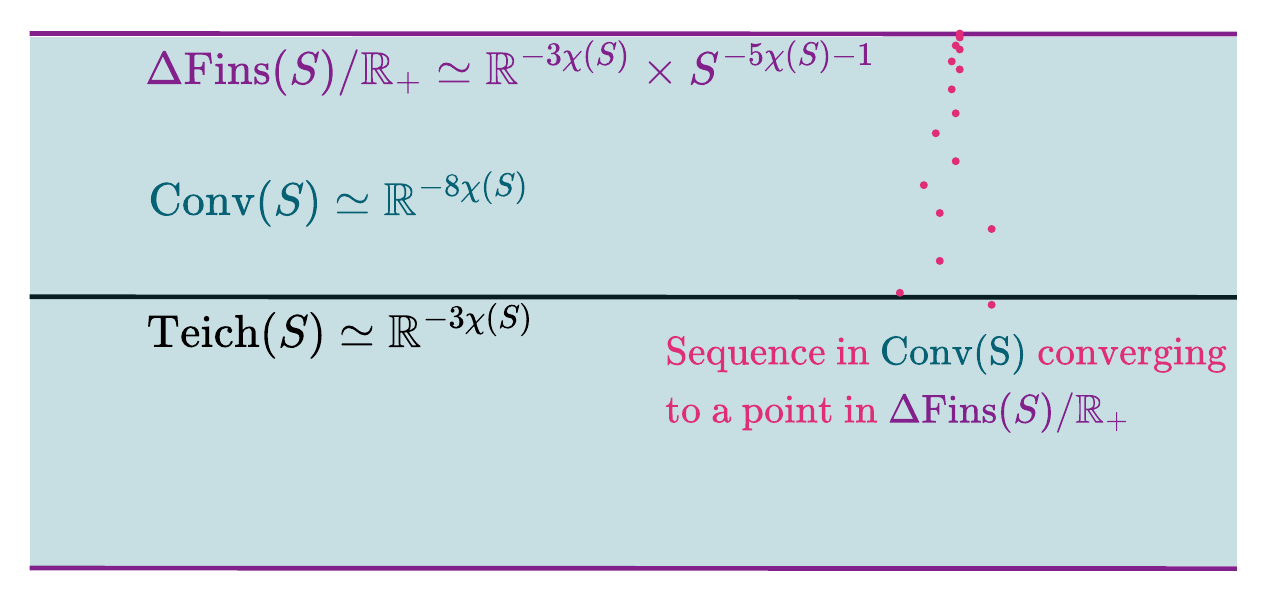}
\end{center}
 Let $\Delta Fins(S)$ denote the set of triangular Finsler metrics $\{F^\Delta_\mu:\mu\in Q(S)\backslash\underline{0}\}$. We conjecture that $\Delta Fins(S)/\RR_+$, injects into $\PP(\RR^{\mathcal{L}(S)})$ and comprises an open dense subset of the boundary of $\operatorname{Conv}(S)$.

\subsection{Relation to other work, and future directions}
This work fits into a few different bigger stories. Firstly, Gaiotto Moore and Neitzke \cite{gaiotto_wall-crossing_2013} have a conjecture for the asymptotics of $tr(\rho(\gamma))$ along rays of Higgs bundles $\{(E,R\phi):R\in \RR_+\}$ under the nonabelian hodge correspondence. Our paper, and \cite{loftin_limits_2022}, can be seen as proving the leading order part of the conjecture for the subspace of the $\operatorname{SL}_3$ Hitchin section where the quadratic differential is zero. The methods of \cite{loftin_limits_2022} seem like they might work in higher rank, while our methods seem very specific to $\operatorname{SL}_3\RR$. On the other hand, our method is less similar to ideas of \cite{gaiotto_wall-crossing_2013} in that Stokes lines don't make an appearence, so it might bring a new perspective to the conjectures.

Our Theorem \ref{metriclimit} is similar in flavor to work of Oyang-Tambourelli \cite{ouyang_limits_2021} where it is shown that when $\mu$ is big, the Blaschke metric is close to the singular flat Riemannian metric defined by $\mu$. They are also able to fully understand the closure of the space of projectivised length spectra of Blaschke metrics as mixed structures on $S$, which consist of a singular flat metric on part of the surface, and a measured lamination on the rest.  We hope that the tropical compactification of $\operatorname{Conv}(S)$ has a similar description, but with singular flat triangular Finsler metrics instead of singular flat Riemannian metrics. In \cite{ouyang_limits_2021}, it is shown that singular flat metrics comprise an open dense subset of mixed structures, motivating our analogous conjecture about singular, flat, triangular Finsler metrics.

In the case of punctured surfaces, Fock and Goncharov \cite{fock_moduli_2006} have explained how to compactify Hitchin components by the projectivization of the tropical points of the character variety. In the $\operatorname{SL}_2$ case, these tropical points parametrize measured laminations, and the integral tropical points correspond to integral laminations. This has led many people to investigate what kind of objects the Fock-Goncharov tropical points are parametrizing in higher rank. 

In \cite{parreau_invariant_2015}, Parreau shows that a certain cone of FG tropical points is parametrizing strucutres on the surface which are part $1/3$ translation structure, and part tree. She shows that asymptotics of Jordan projections of group elements are encoded in a Weyl cone valued metric. For the $1/3$ translation structure part, this is exactly equivalent to the way we encode asymptotics of top eigenvalues with the Finsler metric $F_\mu$. The present paper can be seen as accomplishing, for closed surfaces, something quite analogous to Parreau's work on punctured surfaces. In neither case are all tropical points covered, but the the limitations are not quite the same. Parreau has to restrict to a certain cone of tropical points which depends on the chosen triangulation, but this cone includes some cases which exhibit tree behavior. We, on the other hand, study exactly the cases which do not exhibit tree behavior.

Douglous and Sun \cite{douglas_tropical_2021} have developed a different perspective on the Fock-Goncharov integral $\operatorname{SL}_3$ tropical points, showing that they parametrize certain bipartite trivalent graphs introduced by Sikora-Westbury \cite{sikora_confluence_2007} called non-elliptic webs. Webs and $1/3$ translation structures are related by a simple geometric construction (which has been contemplated by J. Farre. myself, and possibly other people.) Let $S$ be a closed surface. Suppose we have a filling non-elliptic web: a trivalent bipartite graph $W$ embedded in $S$ whose complementary regions are all disks with at least six sides. We can replace each black vertex of $W$ with the equilateral triangle $\operatorname{Conv}(0,1,1/2+i\sqrt{3}/2)\subset \CC$ with cubic differential $dz^3$, and each white vertex with $\operatorname{Conv}(0,1,1/2-i\sqrt{3}/2)\subset \CC$, then glue these triangles according to the edges of the web. This construction produces a surface with cubic differential, which is identified with $S$ up to isotopy. In the other direction, a cubic differential on $S$ with integral periods is the same as a singular $1/3$ translation structure with holonomy valued in $\ZZ/3 \ltimes \ZZ^2\subset \ZZ/3 \ltimes \RR^2$. One can take the preimage of a standard hexagonal web on $\RR^2$ via the developing map to get a filling non-elliptic web on $S$. Below is an example of a web on a genus 2 surface specifying a $1/3$ translation structure glued from $16$ equilateral triangles.
\begin{center}
	\includegraphics[width=8cm]{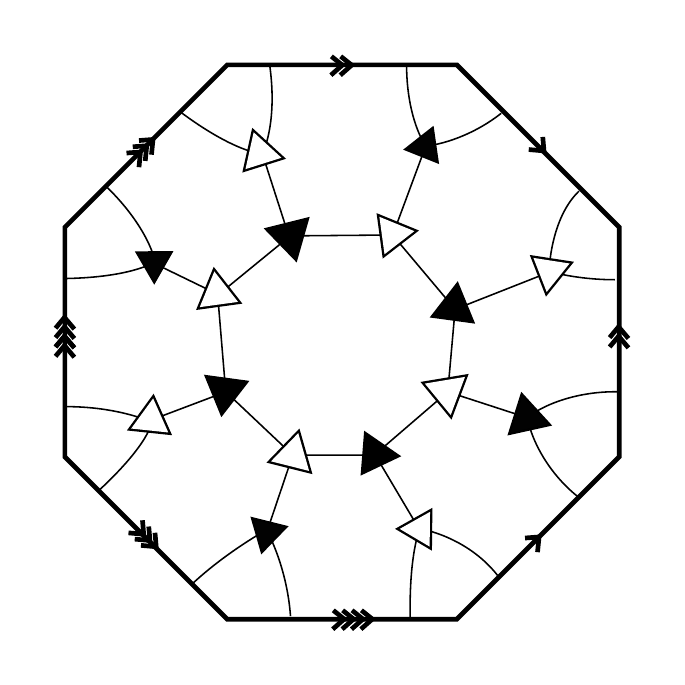}
\end{center}
Using this correspondence, we get a map from filling non-elliptic webs to integral tropical points of $Rep(\pi_1S, \operatorname{SL}_3\RR)$. It would great to extend the results of this paper to punctured surfaces so that they could be compared with work of Parreau, Douglas-Sun, and Fock-Goncharov.

A final, much more wide open, future direction is to use ideas of this paper to understand degenerations of higher dimensional convex projective manifolds. As we will see in the next section, the domain shape metric is defined in any dimension. It is conceivable that one could make sense of limits of domain shape metrics in higher dimensions, and relate these to tropical points of the relevant character varieties. There are various instances throughout geometry where studying more degenerate, combinatorial versions of a class of objects is the key to understanding many important things. Maybe convex projective geometry will be another example of this trend.

\subsection{Structure of the paper}
In Section 2 we define domain shape metrics and prove their important properties. In section 3 we review the Labourie-Loftin correspondence, mostly to recall formulas and set notation. In section 4, the technical heart of the paper, we prove theorem \ref{metriclimit}. In section 5 we prove that length functions of Finsler metrics are sufficiently continuous to deduce theorem \ref{evlimit} from theorem \ref{metriclimit}. In section 6 we explain how to apply theorem \ref{evlimit} to triangle reflection groups, and present some numerical computations. Some readers may prefer to look at section 6 first.

\subsection{Acknowledgements}
This paper would not exist without many discussions with J. Danciger and F. Stecker. In particular, they told me about domain shape metrics, and F. Stecker did various computer experiments giving evidence for Theorem \ref{evlimit}. This research was supported in part by NSF grant DMS-1937215, and NSF grant DMS-1945493. 

\section{Domain shape Finsler metrics}
In this section, we construct a Finsler metric $F^{DS}$ on any convex projective manifold (of any dimension) whose geodesics are projective lines, and whose length spectrum is the assymmetric length spectrum. The construction depends on a lift of the domain in $\mathbb{RP}^n$ to a convex hypersurface in $\RR^{n+1}$. After this section we will specialize to the case where we use the affine sphere as our choice of lift. This metric, and its main properties, were shown to me by Danciger and Stecker, but have not yet appeared in the literature. 

\begin{definition}
Let $v$ be a tangent vector to a point $x$ in a properly convex domain $\Omega\subset \RR \PP^n$. Let $\beta$ be a linear functional defining a supporting hyperplane at the point where the ray starting at $x$ and tangent to $-v$ intersects $\del \Omega$. Let $S\subset \RR^{n+1}$ be a convex, differentiable lift, which is asymptotic to the cone over $\Omega$ in the sense that that the line going through two points in $S$ is never in $\bar{\Omega}$. Let $\tilde{x}$ and $\tilde{v}$ be the lifts of $x$ and $v$ to $S$. 
\[F^{DS}(x,v):= \frac{\beta(\tilde{v})}{\beta(\tilde{x})}\]
We call $F^{DS}$ the domain shape metric of $\Omega$ for the hypersurface $S$.
\end{definition}

The name is justified by the fact that each unit ball of $F^{DS}$ is projectively equivalent to $\Omega$. More specifically, the unit ball at $x\in \Omega$ is the antipodal image of the intersection of the tangent space of $S$ at $\tilde{x}$ with the cone over $\Omega$.

\begin{center}
\includegraphics[width=14cm]{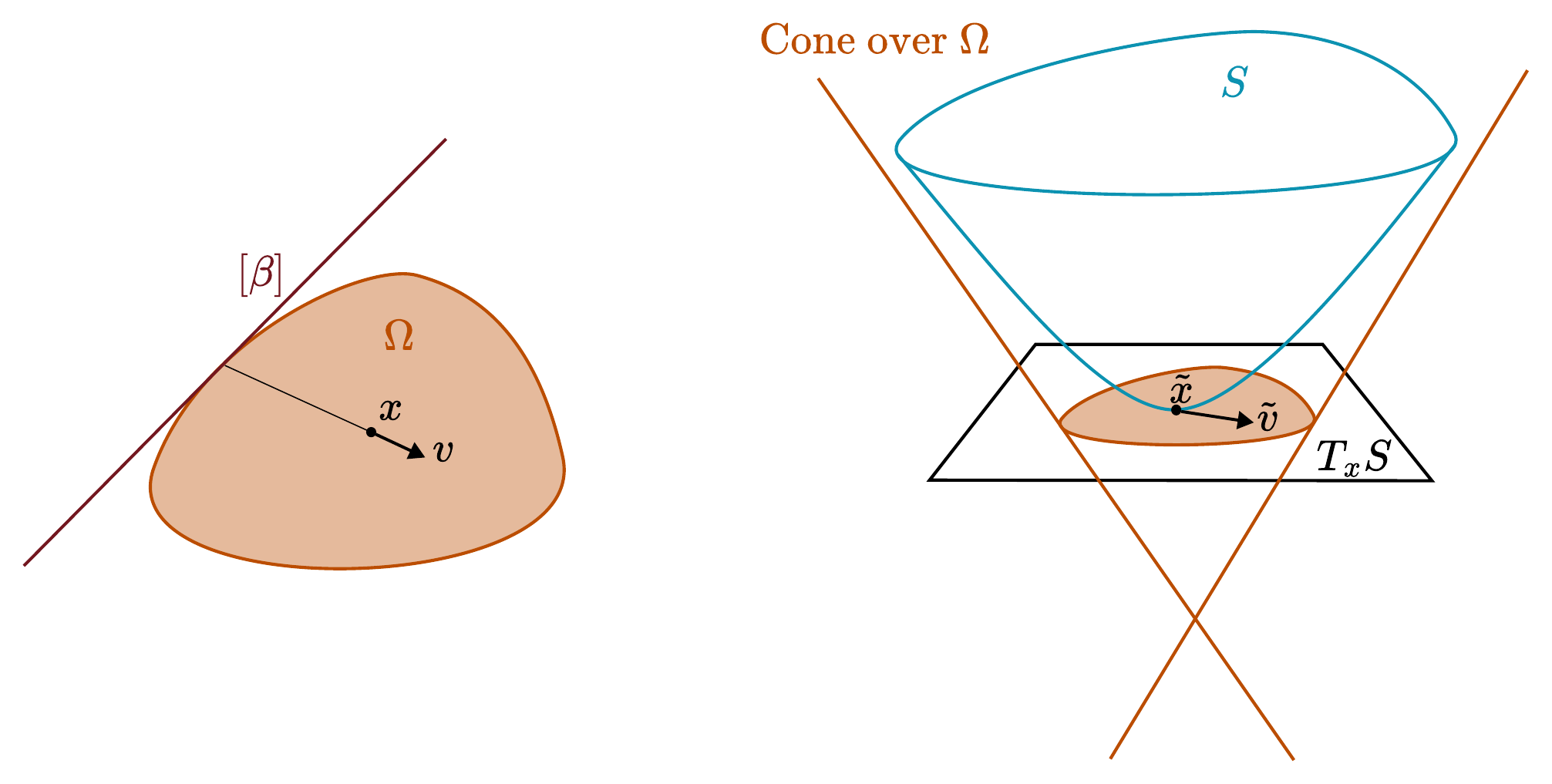}
\end{center}

\begin{remark}
Domain shape metrics are a simple generalization of Funk metrics \cite{funk_uber_1930}, which have been studied since Funk introduced them in 1929, because their geodesics are straight lines. An interpretation of Hilbert's 4th problem is to classify Finsler metrics on euclidean space whose geodesics are straight lines, so there has been interest in such metics for a long time.
\end{remark}

It turns out that $F^{DS}$ integrates to the asymmetric path metric $d^{DS}$ which we now define.

\begin{definition} Let $\Omega$ be a convex domain in $\RR \PP^n$, and let $S$ be a lift of $\Omega$ to $\RR^{n+1}$ which is convex, such the line going through two points of $S$ is never in $\bar{\Omega}$. We define a metric on $\Omega$ as follows. Let $x,y\in \Omega$ be distinct points, and let $\tilde{x},\tilde{y}\in S$ be their lifts. Let $p$ be the point where the projective line starting at $y$ and passing through $x$ first hits the boundary of $\Omega$. Let $[\beta]$ be a supporting hyperplane to $\Omega$ at $p$, defined by a linear functional $\beta$. 
\[d^{DS}(x,y) := \log\frac{\beta(\tilde{y})}{\beta(\tilde{x})}\]
\end{definition}
\begin{center}
\includegraphics[width=8cm]{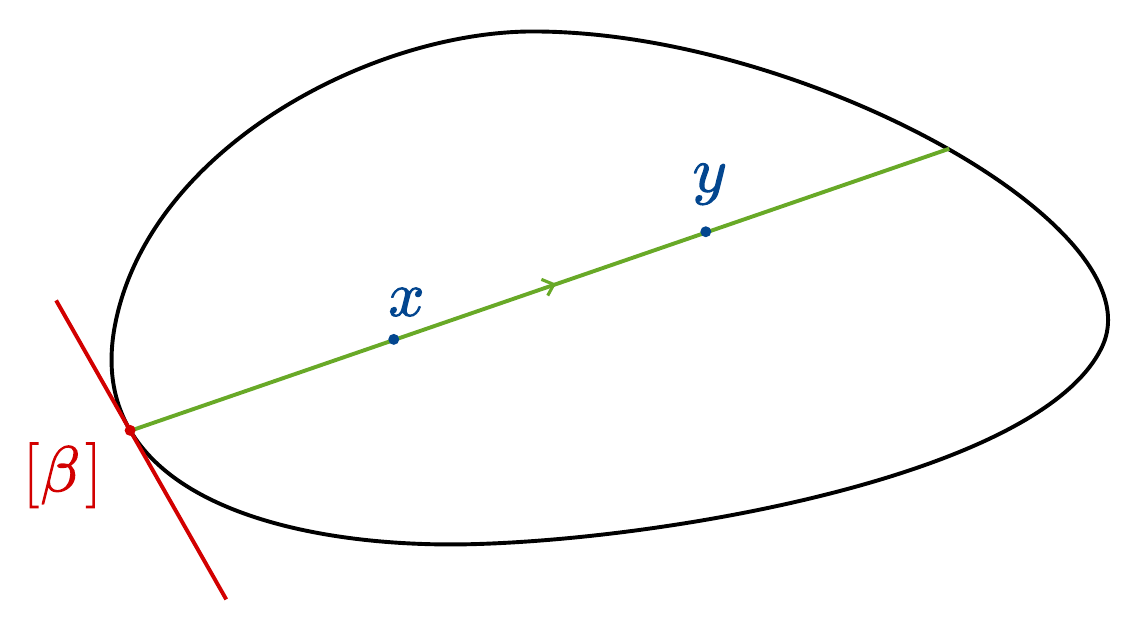}
\end{center}
\begin{remark}
In the case when $\Omega$ has a sharp corner at $p$, there are multiple choices of $[\beta]$, but they give the same value.
\end{remark}

\begin{remark}
Nicolas Tholozan pointed out to me an infinite dimensional case where $d^{DS}$ has been studied. Let $C$ be the cone of parametrizations of the geodisic foliation of the unit tangent bundle of a hyperbolic surface (up to weak conjugacy) and let $S$ be the hypersurface of positive, entropy 1 parametrizations. One can restrict the resulting domain shape metric on $\mathbb{P}(C)$ to spaces of Anosov representations, and get asymmetric Finsler metrics studied in \cite{tholozan_teichmuller_nodate} and \cite{carvajales_thurstons_2022}.
\end{remark}

Before we relate $d^{DS}$ to $F^{DS}$, we check that $d^{DS}$ satisfies the two axioms of an assymmetric metric: reflexivity, and the triangle inequality. Reflexivity is left to the reader. The triangle inequality will follow from the following lemma. 
\begin{lemma} Let $x,y,\Omega, S, \beta$, and $p$ be as above. If $[\beta']$ is any hyperplane which doesn't intersect $\Omega$, and doesn't pass through $p$, then:
\[\frac{\beta'(\tilde{y})}{\beta'(\tilde{x})} < \frac{\beta(\tilde{y})}{\beta(\tilde{x})}\]
\end{lemma}

\begin{proof} Let $p'$ be the point where the line through $x$ and $y$ hits the hyperplane $[\beta']$. We can find an affine hyperplane $A$ going through $\tilde{x}$ and $\tilde{y}$ whose intersection with the cone over $\Omega$ is as follows.

 By assumption, $[\tilde{x}-\tilde{y}]$ cannot be in $\bar{\Omega}$. This means that the dual hyperplane $[\tilde{x}-\tilde{y}]^*$ intersects the dual domain $\Omega^*$. Let $\alpha$ be a linear functional representing a point in the intersection. By construction, $\alpha(\tilde{x})=\alpha(\tilde{y})$, and we can scale $\alpha$ so that $\alpha(\tilde{x})=\alpha(\tilde{y})=1$. Let $A$ be the hyperplane defined by $\alpha=1$.

This affine hyperplane $A$ is identified with an affine chart of projective space containing $\Omega$. Arbitrarily choose a euclidean metric $d_A$ on $A$ compatible with the affine structure. Since $\beta|_A$ and $\beta'|_A$ are affine linear functions vanishing $[\beta]$ and $[\beta']$, they are proportional to the affine linear functions which simply measure signed euclidean distance to $[\beta]$ and $[\beta']$. Our desired inequality is thus equated with an inequality involving euclidean distances which is visually clear.
\[\frac{\beta'(\tilde{y})}{\beta'(\tilde{x})} = \frac{d_A(y,p')}{d_A(x,p')} < \frac{d_A(y,p)}{d_A(x,p)}=\frac{\beta(\tilde{y})}{\beta(\tilde{x})}\] 
\end{proof}
\begin{lemma}
\begin{enumerate}
	
\item $d^{DS}$ satisfies the triangle inequality,
$d^{DS}(x,y) + d^{DS}(y,z) \geq d^{DS}(x,z)$.
\item If $x$, $y$, $z$ are collinear, then we have equality.
\item If the domain is strictly convex, and $x$, $y$, $z$ are not collinear, then the inequality is strict.
\end{enumerate}

\end{lemma}
\begin{center}
\includegraphics[width=8cm]{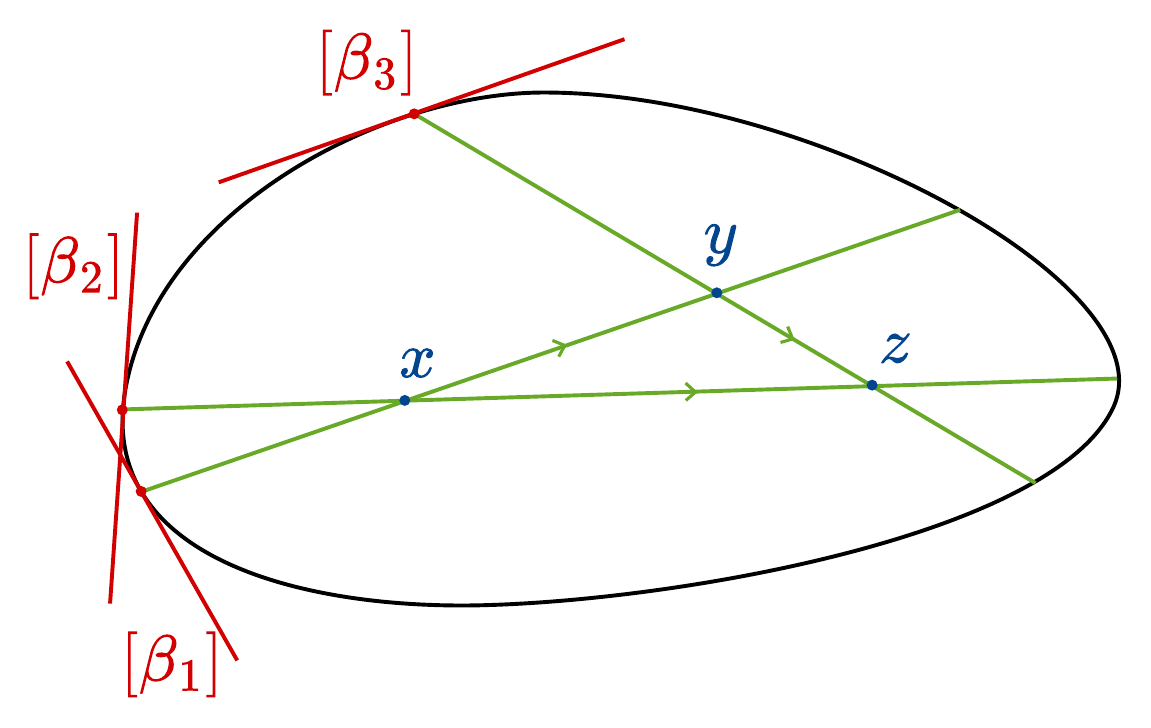}
\end{center}
\begin{proof}
Let $\beta_1,\beta_2$, and $\beta_3$ be linear functionals defining hyperplanes tangent to the points on $\del\Omega$ intersecting the rays $\overrightarrow{yx}$, $\overrightarrow{zy}$, and $\overrightarrow{zx}$ respectively.
\[ \frac{\beta_3(z)}{\beta_3(x)} =  \frac{\beta_3(z)}{\beta_3(y)} \frac{\beta_3(y)}{\beta_3(x)} \leq \frac{\beta_2(z)}{\beta_2(y)} \frac{\beta_1(y)}{\beta_1(x)}\]
Taking logarithms gives the triangle inequality. If $x$, $y$, and $z$ are colinear, then we can choose $\beta_1=\beta_2=\beta_3$, so we have equality. If $x$, $y$, and $z$ are not colinear, and $\Omega$ is strictly convex, then $\beta_1$, $\beta_2$, and $\beta_3$ are supporting hyperplanes at unique, and distinct points, so the previous lemma gives us strict inequality.
\end{proof}

Part $2$ of the preceding lemma implies that projective lines are geodesics of $d^{DS}$, and that $d^{DS}$ is a path metric. Part $3$ implies that if $\Omega$ is strictly convex, projective lines are the only geodesics. Finally, we check that $d^{DS}$ and $F^{DS}$ agree.
\begin{lemma}
$d^{DS}$ differentiates to $F^{DS}$.	
\end{lemma}
\begin{proof}
Let $v$ be a tangent vector at $x\in \Omega$. Let $\gamma:(-\epsilon,\epsilon)\to \Omega$ be a curve with $\gamma'(0)=v$ which we choose for convenience to be a projective line segment. Let $\beta$ be a linear functional as in the definition of $F^{DS}(x,v)$.
\[\frac{d}{dt}d^{DS}(\tilde\gamma(0),\tilde\gamma(t))|_{t=0} = \frac{d}{dt}\log(\frac{\beta(\tilde\gamma(t))}{\beta(\tilde\gamma(0))})|_{t=0} = \frac{\beta(\tilde\gamma'(0))}{\beta(\tilde\gamma(0))}=F^{DS}(x,v)\]
\end{proof}

\subsection{Domain shape metrics for dual hypersurfaces}
In this section we show that, if we choose compatible lifts, domain shape metrics for projectively dual domains are pointwise dual, with respect to a natural Riemannian metric. In the case when the lift is an affine sphere, this metric is known as the Blaschke metric. This will be used in section \ref{lower bounds} to turn upper bounds for $F^{DS}$ into lower bounds.

Let $V$ be a real vector space of dimension $n+1$. Let $\Omega\subset \mathbb{P}(V)$ be a properly convex domain. Let $S\subset V$ be a proper convex lift. Let $\Omega^* \subset \mathbb{P}(V^*)$ be the dual convex domain, and let $S^*\subset V^*$ be the dual convex lift: the set of linear functionals $\alpha$ such that $\inf \alpha(S)=1$. We say that $x\in S$ and $\alpha\in S^*$ are dual points if $\alpha(x) = 1$. If $S$ and $S^*$ are strictly convex, then each point in $S$ has exactly one dual point in $S^*$. Assume that $S$ is smooth, with everywhere positive hessian, so that the duality mapping between $S$ and $S^*$ is a diffeomorphism. In this setting, a natural Riemannian metric appears.

\begin{lemma}
Let $x:\RR^n\to S$ and $\alpha:\RR^n\to S^*$ be smooth, dual parametrizations. The matrix
\[g_{ij} = -\langle \del_i \alpha,\del_j x\rangle = \langle \alpha, \del_i \del_j x\rangle = \langle \del_i \del_j \alpha, x\rangle\]
is a Riemannian metric which doesn't depend on the parametrization.
\end{lemma}
\begin{proof}
	The first expression $-\langle \del_i \alpha,\del_j x\rangle$ is independant of parametrization because we can phrase it in parametrization independant language: if $\beta\in T_\alpha S^*$ and $v\in T_x S$, these both give tangent vectors to $S$, (or $S^*$,) and we define $g(\beta,v)=-\langle \beta, v\rangle$. The second expression $\langle \alpha, \del_i \del_j x\rangle$ is clearly symmetric and positive definite because partial derivatives commute, and $S$ is convex with positive Hessian. It remains to show that these expressions are indeed equal. That $x$ and $\alpha$ are dual means that $\langle \alpha, \del_i x\rangle = \langle \del_i\alpha, x\rangle = 0$ for all $i$. Applying $\del_j$ to these equations gives the result.
\end{proof}

The Blaschke metric is $\langle \del_i \del_j \alpha, v\rangle$ where $v$ is a certain canonically defined normal vector called the `affine normal'. An affine sphere (centered at $0$) is precisely a hypersurface satisfying $v=x$, so for affine spheres $g$ is the Blaschke metric.

\begin{definition}
Let $\Omega$ be a convex subset of a vector space $V$ containing the origin, and let $g$ be a metric on $V$. The dual of $\Omega$ with respect to $g$ is $\Omega^{*_g}:=\{v \in V : g(v,\Omega) < 1\}\subset V$
\end{definition}
If $M$ is a manifold with Riemannian metric $g$, and Finsler metric $F$, let $F^{*_g}$ denote the Finsler metric whose unit balls are pointwise dual to unit balls of $F$ with respect to $g$.

\begin{lemma}
\label{DS duality}
If $S$ is a smooth, properly embedded lift of a properly convex domain $\Omega$ with positive definite Hessian, and we identify $S$ with $S^*$ by the duality diffeomorphism, then $F^{DS}_{S^*} = (F^{DS}_S)^{*_g}$, where $g$ is the metric defined above.
\end{lemma}
\begin{proof}
Let $C$ denote the cone $\RR_+ S$. Let $B$ denote the unit ball of $F^{DS}_S$ at $x\in S$. 
\[B = \{v 
\in T_x S  : x-v\in C\}\]
The unit ball of $F^{DS}_{S^*}$ at $\alpha$ is 
\[B^* = \{\beta 
\in T_\alpha S^*  : \langle \alpha - \beta, C \rangle > 0 \}\]
$C$ is the $\RR_+$ span of $x-B$, so equivalently
\[B^* = \{\beta 
\in T_\alpha S^*  : \langle \alpha - \beta, x-v \rangle > 0 \;\;\;\forall v\in B\}\]
Recall $\langle \alpha,x\rangle=1$, $\alpha$ vanishes on $T_x S$ and $x$ vanishes on $T_\alpha S^*$, so we have $\langle \alpha - \beta, x-v \rangle = 1 + \langle \beta, v\rangle$. Finally we have
\[B^* = \{\beta 
\in T_\alpha S^*  : -\langle \beta, v \rangle < 1 \;\;\;\forall v\in B\}\]
If we identify $T_\alpha S^*$ and $T_x S$ by the differential of the duality map, then $-\langle \beta, v \rangle$ is $g$ evaluated on $\beta$ and $v$, so the unit balls of $F^{DS}_S$ and $F^{DS}_{S^*}$ are dual with respect to the $g$.
\end{proof}

\section{Review of the Labourie-Loftin correspondence}
The Labourie-Loftin correspondence \cite{loftin_affine_2001}, \cite{labourie_flat_2006} provides a bijection between pairs $(J,\mu)$, where $J$ is a complex structure, and $\mu$ is a holomorphic cubic differential, and convex projective structures, on a surface of genus at least $1$. After quotienting both the space of pairs $(J,\mu)$, and the space of convex projective structures by $\operatorname{Diff}_0(S)$, the Labourie-Loftin correspondence becomes a diffeomorphism from the bundle of cubic differentials over Teichmuller space, to the Hitchin component. Here, we review how one gets a convex projective structure on $S$ from a pair $(J,\mu)$. 

Let $S$ be a Riemann surface with cubic differential $\mu$. It turns out that there is a unique hermitian metric $g$ on $S$ satisfying Wang's equation.
\[\kappa_g = |\mu|^2_g - 1\]
From the data $(S,J,\mu,g)$ we will construct a convex projective structure. This means a developing-holonomy pair: a representation $\pi_1(S)\to \operatorname{SL}_3\RR$, and an equivariant embedding $\tilde{S}\to \mathbb{P}(\RR^3)$ whose image is a convex set. We will actually construct an equivariant map $\tilde{S}\to \RR^3$, whose image is a strictly convex hypersurface, (in fact an affine sphere) which we can compose with the projection to $\mathbb{P}(\RR^3)$ to get a developing map. 

The construction of the developing holonomy pair necessitates a choice of base point, and it is useful to pay attention to this choice. For each point $x\in S$, we will construct an affine sphere in the three dimensional vector space $T_x S\oplus \underline{\RR}$ with $\pi_1(x,S)$ action, but different choices of $x$ will give isomorphic results. There is an explicit formula for a real flat connection $\nabla$ on $TS\oplus \underline{\RR}$ in terms of $\mu$ and $g$. We write this formula in terms of the complexification which has a natural line decomposition $TS \oplus \overline{TS} \oplus \underline{\CC}$:
\[
\nabla = 
\begin{bmatrix}
 	D_{TS}^g & g^{-1}\bar{\mu} & 1 \\
 	g^{-1}\mu & D_{\overline{TS}}^g & 1 \\
 	g & g & 0 \\
 \end{bmatrix}
 \]
 Here, $D_{TS}^g$, and $D_{\overline{TS}}^g$ are the Chern connections, which both coincide with the Levi-Civita connection for $g$. The off-diagonal entries are maps of various line bundles. 
 
One checks that $\nabla$ is real, and that its flatness is equivalent to Wang's equation. We sketch this second computation here.

\begin{lemma} $\nabla$ is flat if and only if $\kappa_g = |\mu|^2_g - 1$.	
\end{lemma}
\begin{proof}
This statement can be checked locally. Choose a local holomorphic coordinate $z$ on $S$. Write $g=e^\phi dzd\bar{z}$, and $\mu=\mu_0 dz^3$. We get a frame $\del_z,\del_{\bar{z}},\underline{1}$ of $(TS\oplus\RR)\otimes \CC$. In this frame, we can write $\nabla$ as the de Rahm differential plus a matrix valued $1$-form.
\[\nabla = d + A_1 dz + A_2 d\bar{z}\]
\begin{equation}
\label{connection formula}
A_1=\begin{bmatrix}
 	\del \phi & 0 & 1 \\
 	e^{-\phi}\mu_0 & 0 & 0 \\
 	0 & e^\phi & 0 \\
 \end{bmatrix}dz\;\;\;\;\;\;\;\;\;A_2 = \begin{bmatrix}
 	0 & e^{-\phi}\bar{\mu}_0 & 0 \\
 	0 & \delbar\phi & 1 \\
 	e^\phi & 0 & 0 \\
 \end{bmatrix}d\bar{z}
\end{equation}
The curvature of $\nabla$ is 
\[F_\nabla= (-\del_{\bar{z}} A_1 + \del_z A_2 + \frac12[A_1,A_2])dz\wedge d\bar{z}\]
This comes out to be
\[F_\nabla = \begin{bmatrix}
 	-1 & 0 & 0 \\
 	0 & 1 & 0 \\
 	0 & 0 & 0 \\
 \end{bmatrix}(\del_z\del_{\bar{z}}\phi + \frac12(e^{-2\phi}|\mu_0|^2 - e^\phi))dz\wedge d\bar{z}
\]
Recall that $\del_z\del_{\bar{z}}=\frac14\Delta$ where $\Delta=\del_x^2+\del_y^2$ is the laplacian. Vanishing of $F_\nabla$ becomes a PDE for $\phi$:
\[\frac12\Delta\phi = -e^{-2\phi}|\mu_0|^2 + e^\phi\]
Recalling the formula $\kappa_g = -\frac12 e^{-\phi}\Delta \phi$ for the gauss curvature, we get the coordinate independant form of the equation. 
\[\kappa_g = |\mu|_g-1\]
\end{proof}

The formula for $\nabla$ may seem a bit mysterious, so we mention two perspectives from which one could derive it. First, for a general strictly convex surface $S$ in three dimensional affine space $\RR^3$ (endowed with a translation invariant volume form), there is a canonically defined normal vector field called the affine normal, and a metric called the Blaschke metric. The affine normal lets us identify $T\RR^3|_S$ with $TS\oplus \RR$, and the complex structure induced by the Blaschke metric gives a line decomposition $TS\otimes \CC = TS\oplus \overline{TS}$. If we write the trivial connection on $T\RR^3|_S$ with respect to this line decomposition, it takes a form quite similar to (\ref{connection formula}), but with non-trivial tensors in the last column. The fact that the connection takes the form (\ref{connection formula}) is equivalent to $S$ being a hyperbolic affine sphere. We refer to \cite{loftin_survey_2008} for a survey on affine spheres.

From a completely different point of view, Wang's equation is really a special case of Hitchin's equation: if we use $g$ to identify $\overline{TS}$ with $TS^{*}$, we get a harmonic bundle solving the Hitchin equation for the rank 3 Higgs bundle in the Hitchin section corresponding to the cubic differential $\mu$, and the zero quadratic differential. 

 $TS\oplus \underline{\RR}$ has a natural section, namely $\underline{1}\in \Gamma(\underline{\RR})$. Let $\tilde{S}_x$ denote the universal cover of $S$ based at $x\in S$, constructed explicitly as the collection of pairs $(y,[\gamma])$, where $y\in S$ and $[\gamma]$ is a homotopy class of path from $x$ to $y$. For each point $x\in S$, let $\xi_x:\tilde{S}_x\to T_xS\oplus \RR$ denote the map which takes a point $(y,[\gamma])$ to the parallel transport of $\underline{1}(y)$ back to $x$, along $\gamma$ using the connection $\nabla$. We call $\xi_x$ the affine sphere developing map based at $x$. One can deduce, from the form of $\nabla$, that the image of $\xi_x$ is indeed an affine sphere: the affine normal at $x$ is a fixed scalar multiple of $\xi_x$. 

\begin{center}
\includegraphics[width=12cm]{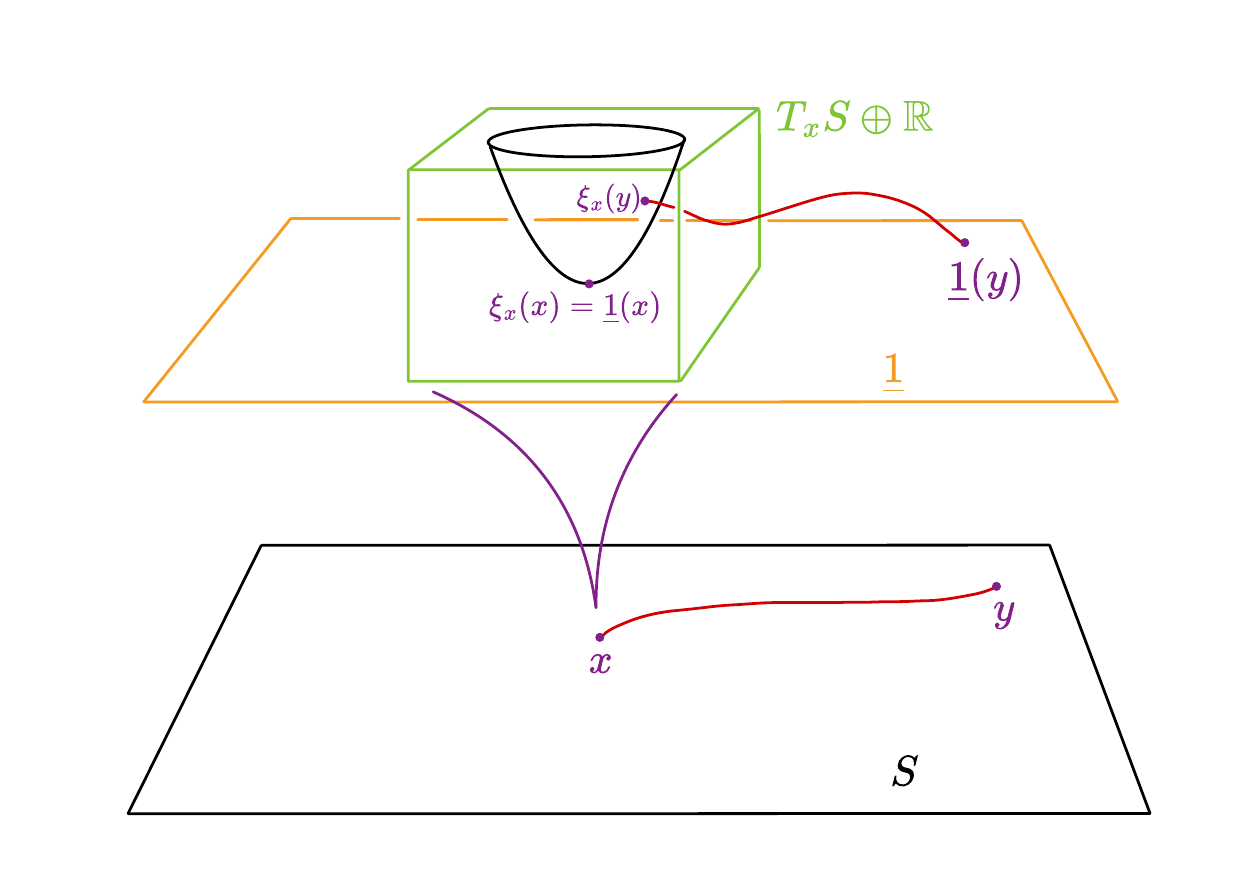}
\end{center}

The fundamental group $\pi_1(S,x)$ acts on $\tilde{S}_x$ via deck transformations, and on $T_x S\oplus \RR$ via the holonomy of $\nabla$. The map $\xi_x$ is equivariant for these two actions by construction. If we choose a different point $x'\in S$, then parallel transport along any path from $x$ to $x'$ will be a volume preserving linear map $T_x S\oplus \RR\to T_{x'} S\oplus \RR$ identifying the affine spheres $\operatorname{Im}(\xi_{x})$ and $\operatorname{Im}(\xi_{x'})$, which intertwines the $\pi_1(x, S)$ and $\pi_1(x', S)$ actions.

\section{$F^{DS}$ is close to $F^\Delta$ far from zeros}
To make precise what we mean by ``far from zeros" we need a metric. A Riemann surface $S$ with cubic differential $\mu$, has a singular, flat Riemannian metric $h$ defined by the equation $|\mu|_{h} = 1$. We will call this singular flat metric $h_\mu$.
\begin{theorem}
\label{metricestimate}
There exists a function $\epsilon:\RR_+ \to (0, \infty]$ with $\lim_{r\to\infty}\epsilon(r) = 0$ such that for any closed Riemann surface $S$ with cubic differential $\mu$, 
\[|\log\frac{F^{DS}_\mu(x,v)}{F^\Delta_\mu(x,v)}| < \epsilon(r(x))\]
for any non-zero tangent vector $v$ at any point $x$, where $r(x)$ denotes the distance from $x$ to the closest zero of $\mu$, with respect to the metric $h$.
\end{theorem}
This theorem is simply a more uniform version of the part of Nie's result \cite{nie_limit_2022} dealing with triangles, and we will prove it following his method.

\subsection{Proof of theorem \ref{metriclimit} assuming theorem \ref{metricestimate}}
We will now prove Theorem \ref{metriclimit} from the introduction as an easy consequence of Theorem \ref{metricestimate}. 

\begin{proof}
Fix $R>0$. Let $S_R \subset S$ be the set of points which are at least distance $R$ away from all zeros of $\mu$, in the metric $h_\mu$. It will suffice to show that $F^{DS}_{\mu_i}/a^i$ converges uniformly to $F^\Delta_\mu$ on $S_R$. Uniform convergence of $a_i^3\mu_i$ to $\mu$ implies that for all $\epsilon>0$, there exists $N$ such that zeros of $\mu_i$ for $i>N$ are all in an $\epsilon$ neighborhood of the zeros of $\mu$ with respect to $h_\mu$. Consequently, the limit of the $h_{\mu}$ distance between the zeros of $\mu_i$ and $S_R$ is $R$. The $h_{\mu_i}$ distance between zeros of $\mu_i$ and $S_R$ must then go to infinity, because $a_i h_{\mu_i}$ converges uniformly to $h_\mu$. This means that the ratio between $F^{DS}_{\mu_i}$ and $F^\Delta_{\mu_i}$ limits to $1$ uniformly on $S_R$. The ratio of $a_i F^\Delta_{\mu_i}$ to $F^\Delta_{\mu}$ also goes to $1$ uniformally on $S_R$. It follows that the ratio of $a^iF^{DS}_{\mu_i}$ to $F^\Delta_\mu$ goes uniformally to $1$ on $S_R$.
\end{proof}

\subsection{Blaschke metric estimate}
Note that the singular flat metric $h_\mu$, defined by $|\mu|_{h_\mu}=1$, is a solution to Wang's equation on the complement of the zeros of $\mu$. The first step in proving theorem \ref{metricestimate} is to show that, far from zeros, the global solution to Wang's equation is close, in $C^1$ norm, to this singular flat solution. 
\begin{lemma}
There exists a function $\epsilon_{C^1}:\RR_+ \to (0,\infty]$ limiting to zero as the input goes to infinity, such that if $S$ is a Riemann surface with holomorphic cubic differential $\mu$ such that $h_\mu$ is complete, $g$ is the complete solution to Wang's equation, and $g=e^\phi h_\mu$, then
\[|\phi(x)| + |d\phi_x|_{h_\mu} < \epsilon_{C^1}(r(x))\]
for all $x\in S$, where $r(x)$ denotes the $h_\mu$ distance from $x$ to the closest zero of $\mu$. 
\end{lemma}
Wang's equation implies that $\phi$ satisfies the following PDE.
\[\frac12\Delta_{h_\mu}\phi = e^\phi - e^{-2\phi}\]
The intuition is that this equation very much wants to force $\phi$ close to zero, and ellipticity can promote $C^0$ bounds to $C^k$ bounds for whatever $k$ we want. For us $C^1$ will be sufficient.
\begin{proof}
In \cite{nie_poles_2023} it is shown that there are uniform constants $C,r_0$ such that $\phi(x)$ is bounded between $0$ and $\epsilon_{C^0}(r(x))$, where:
\[
\epsilon_{C^0}(r)=\begin{cases}
    C\sqrt{r}e^{-\sqrt {6}r}& \text{if } r\geq r_0\\
    \infty              & \text{otherwise}
\end{cases}
\]
This means in particular that on the disk $B_{r(x)/2}(x)$, $\phi$ is bounded between $0$ and $\epsilon_{C^0}(r(x)/2)$. We now use the following interior gradient estimate. 
\begin{lemma}
If $u$ is a twice differentiable function on the closed disk $\overline{B}_x(R)\subset \RR^2$, then we have the following estimate for its gradient at $x$.
\[|\nabla u (x)|\leq R|\Delta u|_{C^0(B_x(R))} + \frac{2}{\pi R}|u|_{C^0(\del B_x(R))}\]
\end{lemma}
This is proved by writing $u$ as a sum of a function which vanishes on the boundary of the disk and has the same laplacian as $u$, and a harmonic function which has the same boundary value as $u$. The gradient of the former can be estimated using its representation in terms of the Green's function for the disk, and the gradient of the latter can be bounded using the maximum principal, and mean value property for harmonic functions. Much more general estimates of this flavor are proved in PDE texts such as \cite{gilbarg_elliptic_2001}.

Apply this to $\phi$ on $B_{x}(r(x)/2)$, recalling that $\Delta \phi = F(\phi)$ where $F(y)=2(e^y-e^{-2y})$.
\[|\nabla\phi(x)| \leq \frac{r(x)}{2}F(\epsilon_{C^0}(r(x)/2)) + \frac{2n}{\pi r(x)} \epsilon_{C^0}(r(x)/2)\]
Note that this gradient estimate goes to zero as $r(x)$ goes to infinity. We simply add the pointwise estimate and the gradient estimate to get the desired $C^1$ estimate.
\[\epsilon_{C^1}(r) := \epsilon_{C^0}(r) + \frac{r}{2}F(\epsilon_{C^0}(r/2)) + \frac{2n}{\pi r} \epsilon_{C^0}(r/2)\]
\end{proof}

\subsection{$F^{DS}$ and $F^\Delta$ coincide for \c{T}i\c{t}eica surfaces}
In this section, we review the fact that the affine sphere corresponding to the constant cubic differential $dz^3$ on the complex plane $\CC$ is a \c{T}i\c{t}eica surface, and we show that the DS metric is exactly the $\Delta$ metric in this case.
\begin{lemma}
\label{titeica coincidence}
	$F^{DS}_{dz^3}=F^\Delta_{dz^3}$
\end{lemma}
\begin{proof}
The standard hermitian metric $h=dzd\bar{z}$ is the complete solution to Wang's equation. By specializing formula (\ref{connection formula}) to the case $\phi=0$, we get the flat connection
\begin{equation}
\label{nabla0 formula}	
\nabla_0 = d +\begin{bmatrix}
 	0 & 0 & 1 \\
 	1 & 0 & 0 \\
 	0 & 1 & 0 \\
 \end{bmatrix}dz
 +\begin{bmatrix}
 	0 & 1 & 0 \\
 	0 & 0 & 1 \\
 	1 & 0 & 0 \\
 \end{bmatrix}d\bar{z}
\end{equation}

 on $T\CC\oplus \underline{\RR}$, expressed in terms of the frame $(\del_z, \del_{\bar{z}}, \underline{1})$ of the complexification $(T\CC\oplus \RR)\otimes \CC$. One directly checks that for each 3rd root of unity $\zeta$,
 \[
 s_\zeta=\begin{bmatrix}
 	\zeta e^{-2Re(\zeta z)} \\
 	\bar{\zeta} e^{-2Re(\zeta z)} \\
 	e^{-2Re(\zeta z)} \\
 \end{bmatrix}
 \]
 is a flat section. We then express $\underline{1}$ in this flat frame.
 \begin{equation}
 	\label{titeica coordinates}
 \underline{1} = \frac{1}{3}\sum_{\zeta \in \sqrt[3]{1}} e^{2 Re(\zeta z)}s_\zeta
 \end{equation}
Identifying all fibers of $T\CC\oplus \RR$ via $\nabla_0$, the flat frame $\{s_\zeta\}$ collapses to the basis $\{(\zeta,1):\zeta\in \sqrt[3]{1}\}\subset \CC\oplus \RR$, and $\frac{1}{3} e^{2 Re(\zeta z)}$ are the coefficients of the affine sphere developing map in this basis.
\begin{center}
\includegraphics[width=8cm]{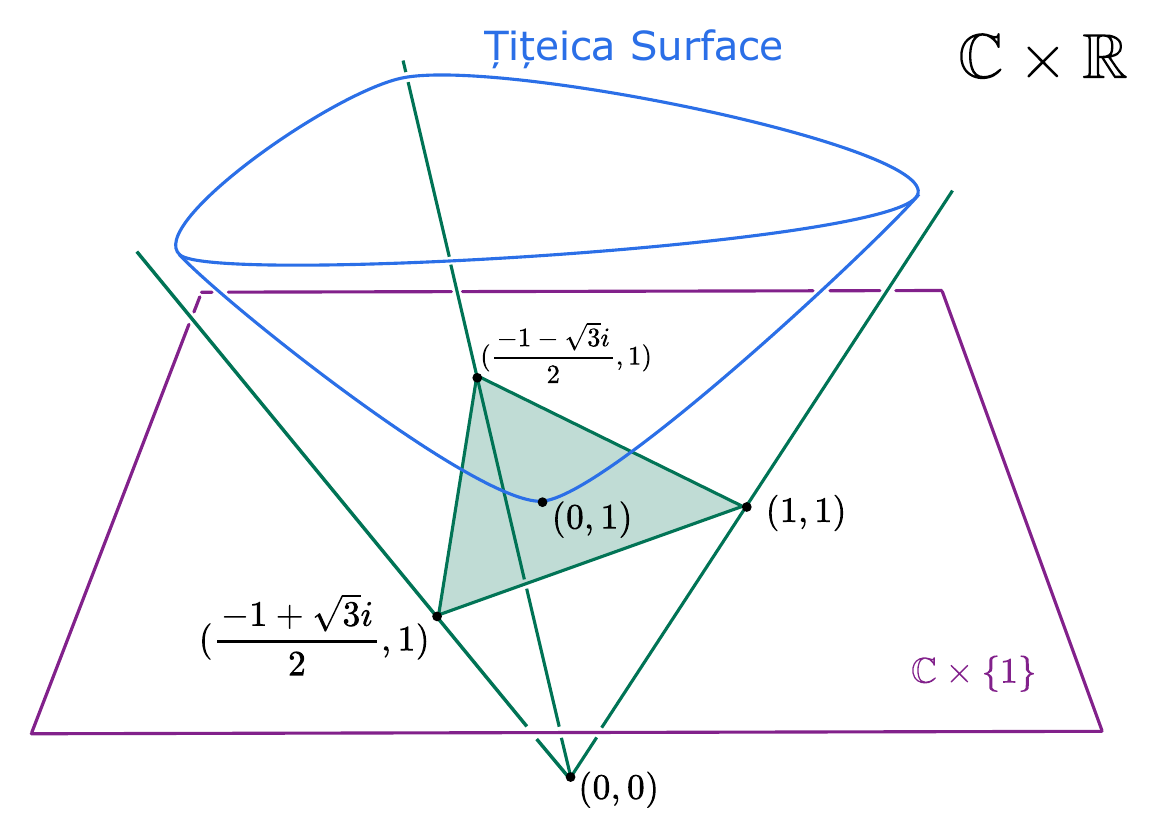}
\end{center}
This particular affine sphere is called a \c{T}i\c{t}eica surface. All we need to notice about it is that its projection to $\CC\times \{1\}$ is indeed the antipodal image of the unit ball for $F^\Delta_{dz^3}$ at $z=0$. Both $F^{DS}_{dz^3}$ and $F^\Delta_{dz^3}$ are translation invariant, so agreement at $z=0$ implies that they are the same everywhere. 
\end{proof}

\subsection{Upper bounds on $F^{DS}$ and $F^\Delta$}
We won't be able to compare $F^{DS}$ and $F^\Delta$ directly. Instead, we will define a families of upper and lower bounds for $F^{DS}$ and $F^\Delta$, and compare these bounds. In this section we define upper bounds $F^{DS,d}>F^{DS}$, and $F^{\Delta,d}>F^\Delta$, and show they are close. In the next section we will use duality arguments to get lower bounds.

The tangent bundle $TS$ embeds into the $\mathbb{RP}^2$ bundle $\PP(TS\oplus \RR)$ as a bundle of affine charts via $v\mapsto [v:1]$. The projectivization of the affine sphere $Im(\xi_x)$, is a convex subset of the affine chart $T_xS \subset \PP(T_xS\oplus \RR)$, and in fact is the antipodal image of the unit ball of $F^{DS}$ at $x$. This suggests a family of upper bounds defined by truncating affine spheres. 

\begin{definition}
	For $d>0$, let $F^{DS,d}$ denote the Finsler metric whose unit ball at $x\in S$ is $-\operatorname{Conv}(\pi(\xi_x(\tilde{B}_{\tilde{x}}(d)))$, where $\tilde{B}_{\tilde{x}}(d)$ denotes the $h_\mu$-ball of radius $d$, centered at $\tilde{x}$ in the universal cover $\tilde{S}_x$, and $\pi$ is the projection from $(T_x S \oplus \RR) \backslash T_x S$ to the affine chart $T_x S\subset \mathbb{P}(T_x S \oplus \RR)$.
\end{definition}

Now we do a similar thing for $F^\Delta$. We let $\mu_0$ and $h_0$ denote the constant cubic differential and hermitian metric on $T_x S$ which agree with $\mu_x$, and $h_x$. Let $\xi_{0,p}: T_x S\to T_x S\oplus \RR$ denote the affine sphere developing map determined by $\mu_0$ and $h_0$.

By lemma $\ref{titeica coincidence}$, the $DS$, and $\Delta$ metrics for the \c{T}i\c{t}eica surface $\xi_{0,x}(T_pS)$ coincide. At the point $x$, this means that the unit ball of $F^\Delta$ at $x$ is $-\pi(\xi_{0,x}(T_pS))$.

\begin{definition}
Let $F^{\Delta,d}$ denote the Finsler metric whose unit ball at $p$ is $-\operatorname{Conv}(\pi(\xi_{0,p}(T^{\leq d}_pS)))$ where $T^{\leq d}_pS$ denotes the ball of radius $d$ in $T_p S$. 
\end{definition}
Note that by construction, $F^{\Delta,d}$ converges to $F^\Delta$. We can thus find a function $\epsilon_{\Delta,d}$, which goes to zero as $d$ goes to infinity, such that 
$\log(F^{\Delta,d}_\mu/F^{\Delta}_\mu) < \epsilon_{\Delta,d}$
for all $d$, for any Riemann surface with cubic differential, on the complement of the zeros.

\subsection{Closeness of truncated affine spheres}
The next step is to show that, far from zeros, $F^{DS,d}$ is close to $F^{\Delta,d}$. This will follow from showing the two affine sphere developing maps $\xi_p$ and $\xi_{p,0}$ are close on the ball of radius $d$. This will give us a bound on the Hausdorff distance between unit balls of $F^{DS,d}$ and $F^{\Delta,d}$. We will then have to prove a simple lemma relating ratios between norms, and Hausdorff distances between their unit balls.

\begin{lemma}
There is a function $\epsilon_{d,Haus}:(0,\infty)\to (0,\infty]$ limiting to zero as the argument goes to infinity, such that for a Riemann surface $S$ with cubic differential $\mu$, the Hausdorff distance, with respect to $h_\mu$, between the unit balls of $F^{DS}_\mu$ and $F^\Delta_\mu$ at any point $p$, is bounded above by $\epsilon_{d,Haus}(r(p))$.
\end{lemma}
\begin{proof}
We define $\epsilon_{d,Haus}(r)$ to be $\infty$ for $r \leq d$, so we can assume that $p$ is distance at least $d$ from zeros. Let $B(p,d)$ denote the ball centered at $p$ of radius $d$ with respect to $h$. We have two solutions to Wang's equation on $B(p,d)$: $g$, and $h$. These give rise to two affine sphere developing maps $\xi_p, \xi_{0,p}: B(p,d) \to T_p S\oplus \RR$ which are constructed by parallel transporting the section $\underline{1}$ back to the fiber over $p$, via two different connections $\nabla_0$, and $\nabla$. As we did for $\nabla_0$ in equation \ref{nabla0 formula}, we can write an explicit formula for $\nabla$ by choosing a coordinate $z$ on $B(p,d)$ which takes $\mu$ to $dz^3$, and using the frame $\del_z, \delbar_z, \underline{1}$ of the complexification of $TS\oplus \RR$. In this frame, we have the following formula for $\nabla$ in terms of $\phi = \log(g/h)$.
 \[
\nabla = d+\begin{bmatrix}
 	0 & 0 & 1 \\
 	e^{-\phi} & \del \phi & 0 \\
 	0 & e^\phi & 0 \\
 \end{bmatrix}dz
 +\begin{bmatrix}
 	\delbar\phi & e^{-\phi} & 0 \\
 	0 & 0 & 1 \\
 	e^\phi & 0 & 0 \\
 \end{bmatrix}d\bar{z}
\]
Let $A_0$ and $A$ be the matrix valued $1$ forms representing $\nabla_0$ and $\nabla$ in this frame: $\nabla_0=d+A_0$, $\nabla=d+A$. 
To conclude closeness of $\xi_p$ and $\xi_{0,p}$ from closeness of $\nabla$ and $\nabla_0$, we need the following standard consequence of Gronwall's inequality. 

\begin{lemma}
Suppose $f' = Af$ and $g'= B g$ where $A,B\in C^0([0,t],\RR^{n\times n})$ are matrix valued functions, and $f,g\in C^1([0,t],\RR^n)$ are vector valued functions with $f(0)=g(0)$. Then $|g-f|$ has the following bound.
\[|g(t)-f(t)|\leq |B-A|_{C^0}|f|_{C^0} t e^{t|B|_{C^0}}\]
\end{lemma}

Let $q\in B(p,d)$. Applying Gronwalls inequality to the restriction of $\nabla$ and $\nabla_0$ to the straight line segment from $q$ to $p$ gives the following.
 \[|\xi_p(q)-\xi_{0,p}(q)|\leq |A-A_0|_{C^0}|\xi_{0,p}|_{C^0} de^{d|A|_{C^0}}\]
 $|A-A_0|$ is bounded by $C\epsilon_{C^1}(r)$ for some fixed constant $C$. By equation \ref{titeica coordinates}, $|\xi_{0,p}|$ is bounded by $\frac13 e^{4d}$. $|A|$ is bounded by $|A_0|+C\epsilon_{C^1}(r)$. We get
 \[|\xi_p(q)-\xi_{0,p}(q)|\leq C \epsilon_{C^1}(r) \frac13 e^{4d} d e^{d(|A_0|+C\epsilon_{C^1}(r))}\]
Define $\epsilon_d(r)$ to be the right hand side of this inequality, and note that it indeed goes to zero as $r$ goes to infinity for any fixed $d$ which is large enough that $\epsilon_{C^1}$ is finite. This bound persists after projecting to $T_p S$, and negating to get the unit balls of $F^{\Delta,d}$, and $F^{DS,d}$. This is because, by convexity, $\xi_0$ and $\xi$ are valued in $\{(v,y)\subset TS\oplus \RR \;|\; y\geq 0\}$, a region in which the radial projection $\pi$ to $T_p S$ is contracting. This bound on $|\pi(\xi_p(q))-\pi(\xi_{0,p}(q))|$ gives the same bound on the Hausdorff distance between unit balls of $F^{DS,d}(p)$ and $F^{\Delta,d}(p)$. 
\end{proof}

\begin{definition}
Let $\Omega_1,\Omega_2\subset \RR^n$ be convex sets containing the origin. We can describe $\Omega_1$ and $\Omega_2$ as the unit balls for norms $f_1$ and $f_2$ on $\RR^n$. Let the radial distance, $d_R(\Omega_1,\Omega_2)$ denote the supremum of $|\log(f_1/f_2)|$. 
\end{definition}
Note that $f_1/f_2$ is the scaling factor which takes the boundary of $\Omega_1$ onto the boundary of $\Omega_2$.

\begin{lemma}
\label{hausdorff-radial}
Let $\Omega_1,\Omega_2\subset \RR^n$ be convex sets containing the origin. Let $d_H(\Omega_1,\Omega_2)$ denote the Hausdorff distance. We have
\[rd_R(\Omega_1,\Omega_2) \leq d_H(\Omega_1,\Omega_2)\]
where $r$ is the radius of a ball centered at $0$ contained in both $\Omega_1$ and $\Omega_2$.
\end{lemma}
\begin{proof}
For $x\in \del \Omega_1$, let $\alpha(x) := f_1(x)/f_2(x)$ be the scaling factor necessary such that $\alpha x\in \del\Omega_2$. Let $x\in \del{\Omega_1}$ be a point realizing the supremum of $|\log(\alpha(x))|$ and, without loss of generality, assume $\alpha(x)\geq 1$. Suppose we have a supporting hyperplane $H_x$ for $\Omega_1$ at $x$. We argue that $H_y=\alpha(x) H_x$ must be a supporting hyperplane for $\Omega_2$ at $y$. Indeed, suppose that there is $y'\in \Omega_2$ on the other side of $H_y$ from $0$. Let $x'$ be the point where the line from $0$ to $z$ intersects $\del \Omega^1$, one sees that $\alpha(x')>\alpha(x)$, a contradiction. It follows that there is no $y'$ on the other side of $H_y$, so $H_y$ is a supporting hyperplane for $\Omega_2$. The Hausdorff distance is at least the euclidean distance from $y$ to $\Omega_1$ which is at least the distance between $H_y$ and $H_x$, which is at least $(\alpha-1)r$. 
\[d_H(\Omega_1,\Omega_2) \geq (\alpha-1)r \geq \log(\alpha)r \geq rd_R(\Omega_1,\Omega_2)\]
\end{proof}

If we assume $d$ is above some threshold $d_0$, we then the unit balls of $F^\Delta_\mu$ will contain the unit balls of $r_0^{-1} h_\mu$ for some $r_0$. Setting $\epsilon_d(r) = \infty$ for $d\leq d_0$, and $\epsilon_d(r) = r_0^{-1}\epsilon_{d,Haus}(r)$ for $d>d_0$ we get our desired control over the ratio between $F^{\Delta,d}$ and $F^{DS,d}$.
\begin{lemma}
There is a function $\epsilon_d:(0,\infty)\to (0,\infty]$ limiting to zero as the argument goes to infinity, such that for a Riemann surface $S$ with cubic differential $\mu$,
\[|\log \frac{F^{DS,d}_\mu}{F^{\Delta,d}_\mu}|\leq \epsilon_d(r(p))\]
\end{lemma}

\subsection{Lower bounds on $F^{\Delta}$ and $F^{DS}$}
\label{lower bounds}
Next, we will construct a family of lower bounds, using  projective duality ideas. 

\begin{lemma}
	On a Riemann surface with cubic differential $\mu$, inducing flat metric $h$, we have \[(F^{\Delta}_\mu)^{*_{2h}} = F^{\Delta}_{-\mu}\] on the complement of zeros.
\end{lemma}
\begin{proof}
Around any point, there is a local coordinate $z$ such that $\mu=dz^{3}$, so it suffices to treat the case of of $\mu=dz^3$ on the complex plane. This is an easy computation. Alternatively, it is a special case of the next lemma.
\end{proof}

When we have a Riemann surface with cubic differential $\mu$, and $g$ satisfying $\kappa_g = |\mu|_g^2 -1$, the Blaschke metric for the corresponding affine sphere is $2g$. The dual affine sphere is given by replacing $\mu$ with $-\mu$ in the formula for the connection. It follows from lemma \ref{DS duality} that negating the cubic differential corresponds to dualizing the domain shape metric with respect to $2g$.

\begin{lemma}
If $S$ is a Riemann surface with cubic differential $\mu$, and $g$ is a complete solution to Wang's equation, then $(F^{DS}_\mu)^{*_{2g}} = F^{DS}_{-\mu}$ where $g$ is the complete solution to Wang's equation. 
\end{lemma}

Equivalently, we have $F^{DS}_\mu = (F^{DS}_{-\mu})^{*_{2g}}$. Since taking duals reverses containment of convex sets, $F^{DS}_{-\mu} < F^{DS,d}_{-\mu}$ implies $F^{DS}_{\mu} > (F^{DS,d}_{-\mu})^{*_{2g}}$. This is our desired lower bound. We will need to show that it is close to $F^{\Delta}_\mu$ just like our upper bound. In the last subsection we showed that $F^{DS,d}$ and $F^{\Delta,d}$ are close. The following lemma implies that their duals, with respect to $2g$, are the same amount close. 

\begin{lemma}
\label{radial-dual}
Taking dual convex sets is an isometry for $d_R$.
\[d_R(\Omega_1,\Omega_2) = d_R(\Omega_1^*,\Omega_2^*)\]
\end{lemma}
\begin{proof}
For $x\in \del \Omega_1$, let $\alpha(x) := f_1(x)/f_2(x)$ be the scaling factor necessary such that $\alpha x\in \del\Omega_2$. Let $x\in \del{\Omega_1}$ realize the supremum of $\alpha$. By definition, $d_R(\Omega_1,\Omega_2) = |\log(\alpha)|$. Without loss of generality, assume $\alpha\geq 1$. As in the proof of Lemma \ref{hausdorff-radial}, if $H_x$ is a supporting hyperplane for $\Omega_1$ at $x$, then $H_y$ is a supporting hyperplane for $\Omega_2$.

Note that $\del\Omega_1^*$ is identified with the set of supporting hyperplanes of $\Omega_1$. If $H$ is a supporting hyperplane for $\Omega_1$, then let $\beta(H)$ be the positive number such that $\beta(H)H$ is a supporting hyperplane for $\Omega_2$. $d_R(\Omega_1^*,\Omega_2^*)$ is the supremum of $|\log(\beta)|$. We have shown that $d_R(\Omega_1,\Omega_2)\leq d_R(\Omega_1^*,\Omega_2^*)$. The reverse inequality follows from the fact that taking dual convex sets is an involution.
\end{proof}

\subsection{End of proof of theorem C}
We have an upper bound on $F^{DS}_\mu$ which is close to $F^\Delta_\mu$
\[F^{DS}_\mu\: \leq \:F^{DS,d}_\mu\:\underset{\epsilon_{d}(r)}{\approx} \: F^{\Delta,d}_\mu\: \underset{\epsilon_{\Delta,d}}{\approx} \: F^\Delta_\mu\]
and a lower bound on $F^{DS}_\mu$ which is also close to $F^\Delta_\mu$
\[F^{DS}_\mu = (F_{-\mu}^{DS})^{*_{2g}} \:\geq \:
(F^{DS,d}_{-\mu})^{*_{2g}}\:\underset{\epsilon_{d}(r)}{\approx} \: 
(F^{\Delta,d }_{-\mu})^{*_{2g}}\:\underset{\epsilon_{C^0}(r)}{\approx} \: 
(F^{\Delta,d }_{-\mu})^{*_{2h}}\:\underset{\epsilon_{\Delta,d}}{\approx} \: 
(F^\Delta_{-\mu})^{*_{2h}} = F^\Delta_\mu\].
Each `$\approx$' symbol means there is a bound, named in the subscript, on the absolute value of the log of the ratio.  Combining upper and lower bounds gives bound on the distance between $F^{DS}_\mu$ and $F^{\Delta}_\mu$:

\[|\log \frac{F^{DS}}{F^{\Delta}}| \leq \epsilon_{\Delta,d} + \epsilon_d(r) + \epsilon_{C^0}(r)\]

We have that $\epsilon_{\Delta,d}$ limits to zero as $d$ goes to infinity, and for each fixed $d$, $\epsilon_d(r)$ goes to zero as $r$ goes to infinity. We just need to choose a function $d(r)$ which limits to infinity, but slowly enough such that $\epsilon_{d(r)}(r)$ goes to zero. Setting $\epsilon(r) = \epsilon_{\Delta,d(r)} + \epsilon_{d(r)}(r) + \epsilon_{C^0}(r)$ finishes the proof.

\section{Continuity of length functions of Finsler metrics}
Depending on context, there are minor variations on what one means by a Finsler metric. So far, we have discussed specific Finsler metrics, so we haven't needed to specify a class of Finsler metrics to work with, but now we need to prove a general fact about Finsler geometry, so we will specify exactly what we mean by Finsler metric.
\begin{definition}
A Finsler metric on a differentiable manifold $M$ is a continuous function $F:TM\to \RR_{\geq 0 }$ satisfying 
\begin{itemize}
\item $F_x(\lambda v) = \lambda F_x(v)$ for all $\lambda \in \RR_{\geq 0 }$
\item $F_x(v+v')\leq F_x(v) + F_x(v')$
\item $F_x(v) = 0$ iff $v=0$
\end{itemize}
\end{definition}
Importantly, we don't assume $F_x(v) = F_x(-v)$, and we don't assume that $F$ is differentiable. If $F$ satisfies the first two conditions, but not the third, we call it a degenerate Finsler metric, and we call the subset of $M$ where $F$ is degenerate the degeneracy locus. Let $\mathcal{F}(M)$ denote the set of Finsler metrics on $M$, and let $\mathcal{F}^{\textrm{fd}}(M)$ denote the set of Finsler metrics with finite degeneracy locus. These definitions are useful to us because if $S$ is a convex projective surface, $F^{DS}$ is always in $\mathcal{F}(M)$, and $F^{\Delta}$ is always in $\mathcal{F}^{\textrm{fd}}(S)$.

If $L$ is a free homotopy class of loop in $M$, and $F\in \mathcal{F}(M)$, then we denote by $F(L)$ the infimum of the lengths of paths in $L$ with respect to $F$.
\[F(L) := \inf_{\gamma \in L}\int_{\gamma}F\]
We give $\mathcal{F}(M)$ the topology in which $F_i$ converge to $F$ if $F_i/F$ converge uniformly to $1$ as functions on $TM- \underline{0}$. It is easy to see that $F(\gamma)$ is a continuous function of $F$ with respect to this topology. We need to be a little more thoughtful in choosing a topology for $\mathcal{F}^{\textrm{fd}}(M)$

\begin{definition}
We endow $\mathcal{F}^{\textrm{fd}}(M)$ with the topology in which $F_i$ converges to a degenerate Finsler metric $F$ with degeneracy locus $X$ if for every open neighborhood $U$ of $X$, $F_i/F$ converges uniformly to $1$ in $M\backslash U$.
\end{definition}

\begin{theorem}
If $F_i\in \mathcal{F}^{\textrm{fd}}(M)$ is a sequence converging to a Finsler metric $F$ with finite degeneracy locus $X$, and $Y$ is a free homotopy class of loop in $M$, then
\[\lim_{i\to\infty} F_i(Y) = F(Y)\]
\end{theorem}

\begin{proof}
First we show 
\[\lim_{i\to \infty}F_i(Y) \leq F(Y)\]
Let $\epsilon>0$. We can choose $\gamma\in Y$ which doesn't hit $X$ such that $F(\gamma) \leq F(Y)+\epsilon$. This follows from continuity of $F$, because only very small deformations of a path are necessary to avoid $X$, and these will change the path's length by a small amount. Because $\gamma$ avoids the degeneracy locus, we get the following:
\[\lim_{i\to \infty} F_i(Y)\leq \lim_{i\to \infty} F_i(\gamma) = F(\gamma) \leq F(Y)+\epsilon\]
Since this holds for all $\epsilon$, it follows that $\lim F_i(\gamma) \leq F(Y)$.

Now we show the reverse inequality. 
\[F(Y) \leq \lim_{i\to \infty} F_i(Y)\]
For every $i\in \NN$, let $\gamma_i\in Y$ be a path which nearly realizes minimal length in the metric $F_i$. 
\[F_i(\gamma_i) \leq F_i(Y) + \epsilon\]
Choose $\delta>0$ such that closed $\delta$-balls around points in $X$, for the metric $F$ are disjoint. Let $d_{\min}$ be shortest distance between two balls.  If a path in the universal cover $\gamma:[0,1]\to \tilde{M}$ hits $N$ different balls, then $F(\gamma)\geq N d_{\min}$. Let $K>1$. For sufficiently large $i$, we have  $K^{-1} < F_i/F < K$ on the complement of the balls. We get a bound on the number of balls a path can visit in terms of its $F_i$ length: \[N \leq K F_i(\gamma)/d_{\min}\] For sufficiently large $i$, $F_i(Y)\leq F(Y) + \epsilon$, so $F_i(\gamma)\leq F(Y) + 2\epsilon$. Let $\tilde\gamma_i:[0,1]\to \tilde{M}$ be a lift of $\gamma_i$ to the universal cover. For simplicity, choose $\tilde\gamma$ so that that $\tilde\gamma(0)$ is not in the $\delta$ neighborhood of $X$. Putting things together, we get a bound on the number of balls $\tilde\gamma_i$ can visit which holds for all sufficiently large $i$.
\[N_i\leq K (F(Y)+2\epsilon)/d_{\min} \]

We now define a path $\gamma_i'$ whose $F$ length we can estimate. Let $B_1$ be the first ball that $\gamma_i$ touches, and let $t_0,t_1\in [0,1]$ be the first and last points which are sent by $\tilde\gamma_i$ to $\overline{B_1}$. Replace $\tilde\gamma_i|_{[t_0,t_1]}$ with a path that goes straight to $p$, and straight out. Now apply the same procedure to the rest of the path $\tilde\gamma_i|_{[t_1,1]}$. Proceed inductively, and call the final result $\tilde\gamma_i'$. We have an upper bound on the $F$ length of $\tilde\gamma_i'$, thus an upper bound on $F(Y)$.
\[F(Y) \leq F(\gamma') \leq K F_i(\gamma_i) + 2N_i\delta \leq K (F_i(Y)+\epsilon) + 2N_i\delta\]
This bound holds for sufficiently large $i$ for any choices of $K$ and $\delta$, and $N_i$ eventually has a uniform bound, so we have $F(Y) \leq \lim_{i\to \infty} F_i(Y)+\epsilon$. This holds for any $\epsilon$, so we get the desired inequality.

\end{proof}

\begin{center}
\includegraphics[width=14cm]{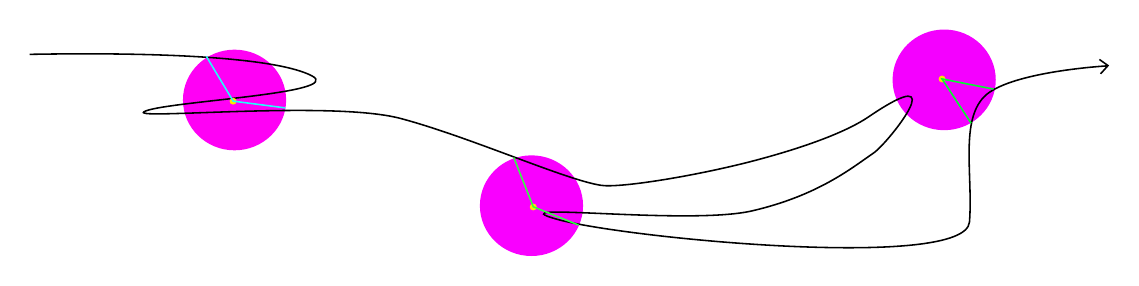}
\end{center}

\section{Application to a triangle reflection group}
The broad aspiration of this project is to understand what happens when we have a sequence of Hitchin representations of a surface group into $\operatorname{SL}_3\RR$ going to infinity. Replacing the surface group with a triangle reflection group is a great way to probe this question, because the Hitchin component is diffeomorphic to $\RR$, so there are only two ways to go to infinity. Before \cite{loftin_limits_2022} was published, and before we knew of Loftin's work \cite{loftin_flat_2007}, it was triangle group computations which convinced us that theorem \ref{evlimit} should hold.

Consider a triangle group.
\[\Gamma_{pqr} := \langle a,b,c \:|\: a^2 = b^2 =  c^2 = (ab)^p = (bc)^q = (ca)^r = 1 \rangle \]
For $1/p + 1/q + 1/r < 1$, there is a unique conjugacy class of homomorphism $\rho_0:\Gamma_{pqr}\to SO(2,1)$ giving a proper discontinuous action on the hyperbolic plane $\HH^2$, and the quotient $S=\Gamma\backslash\HH^2$ is an orbifold. Let $\operatorname{Conv}(\Gamma)$ denote the component of $\operatorname{Rep}(\Gamma, \operatorname{SL}_3\RR)$ containing $\rho_0$.  

We can easily make an explicit algebraic parametrization of $\operatorname{Conv}(\Gamma_{pqr})$. Let $v_1,v_2,v_3$ and $\alpha_1,\alpha_2,\alpha_3$ be bases of $\RR^3$ and $(\RR^3)^*$ with the following matrix of pairings.
\[\alpha_i(v_j) = \begin{bmatrix}
 2 & -2\cos(\frac{\pi}{p})t & -2\cos(\frac{\pi}{r}) \\
 -2\cos(\frac{\pi}{p})t^{-1} & 2 & -2\cos(\frac{\pi}{q}) \\
 -2\cos(\frac{\pi}{r}) & -2\cos(\frac{\pi}{q}) & 2 \\
 \end{bmatrix}\]
We define $\rho_t:\Gamma_{pqr}\to \operatorname{SL}_3\RR$ to send the generators $a$, $b$, $c$ to the three reflections $I - v_i\otimes \alpha_i$. The parameter $t$ is the square root of the triple ratio of the three reflections. It has the following expression:
\[t^2=\frac{\alpha_1(v_2)\alpha_2(v_3)\alpha_3(v_1)}{\alpha_1(v_3)\alpha_2(v_1)\alpha_3(v_2)}\]
which makes it clear that it is an invariant of the three reflections. By \cite{lee_anosov_2021}, we know that $t$ gives a global parametrization of $\operatorname{Conv}(\Gamma_{pqr})$ by $\RR_+$. 

A naive way to try to understand the representation $\rho_t$ is to plot the eigenvalues of $\rho_t(g)$ for a bunch of $g\in \Gamma$. Let $\phi:\operatorname{SL}_3\RR\to \RR^2$ be the function $(\log|\lambda_1|,-\log|\lambda_3|)$ where $\lambda_1$ is the top eigenvalue, and $\lambda_3$ is the bottom eigenvalue. $\phi$ is called the Jordan projection for $\operatorname{SL}_3\RR$ and lands in the cone spanned by $(1,2)$ and $(2,1)$. In general, the Jordan projection of an element of a semi-simple lie group is the Weyl-chamber valued translation length of the the element acting on the associated symmetric space, so Jordan projections are a generalization of hyperbolic translation length.

As an example, we plot here $\phi(g)$ for orientation preserving elements $g\in\Gamma_{444}$ of length at most 24, for triple ratio $t^2$ set to $1,10,1000$, and $10^{12}$.

\includegraphics[width=8cm]{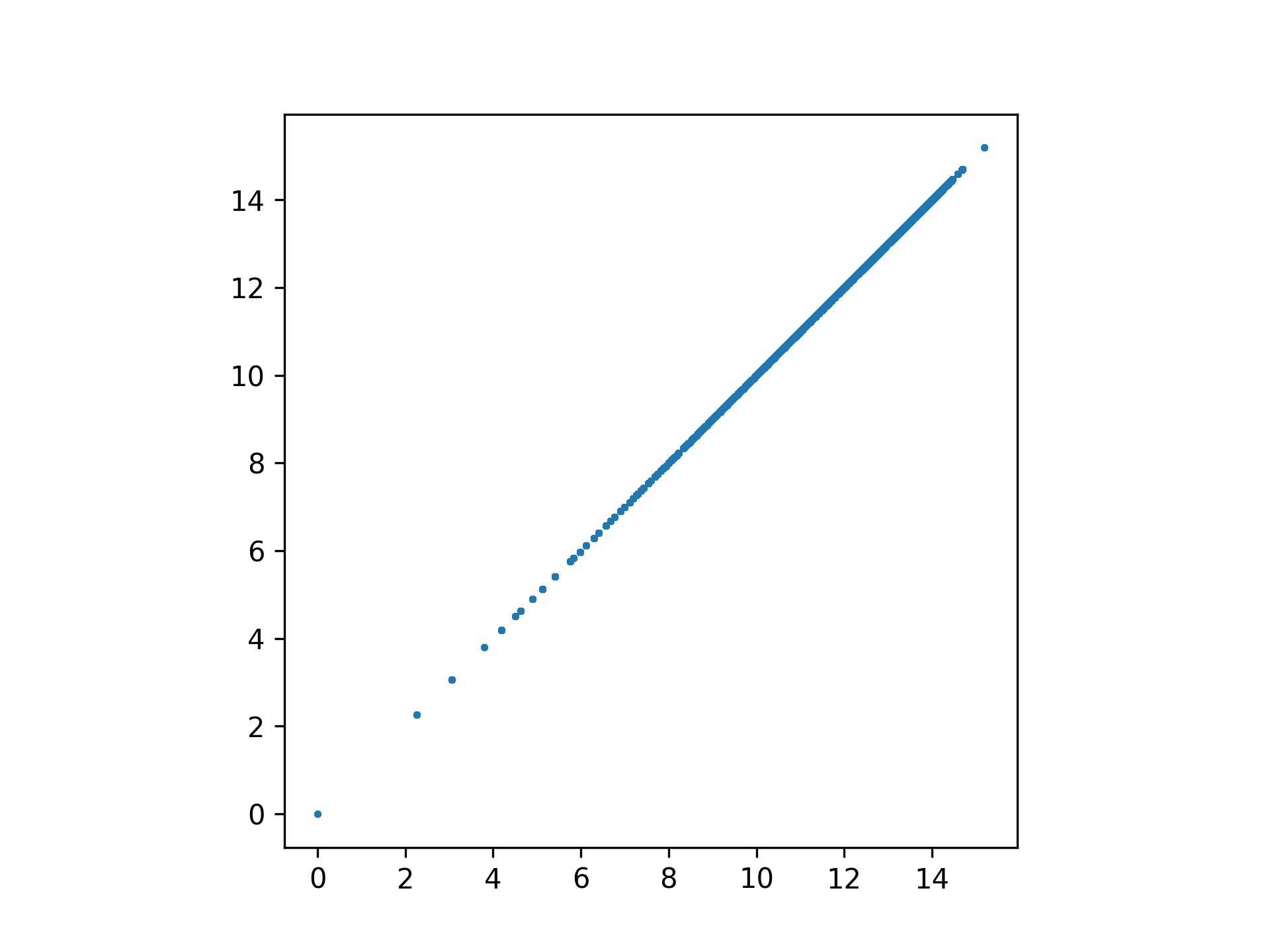}
\includegraphics[width=8cm]{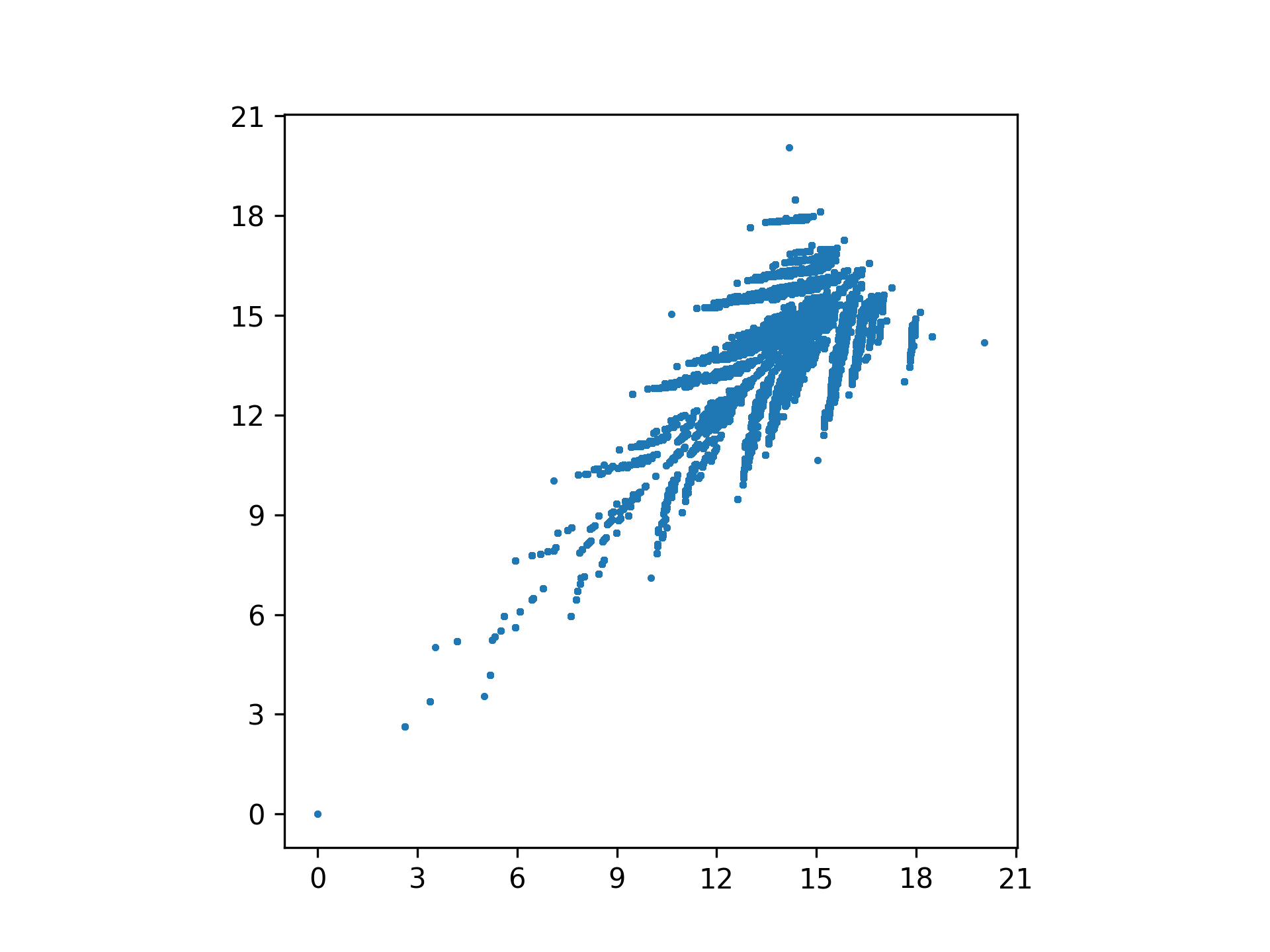}

\includegraphics[width=8cm]{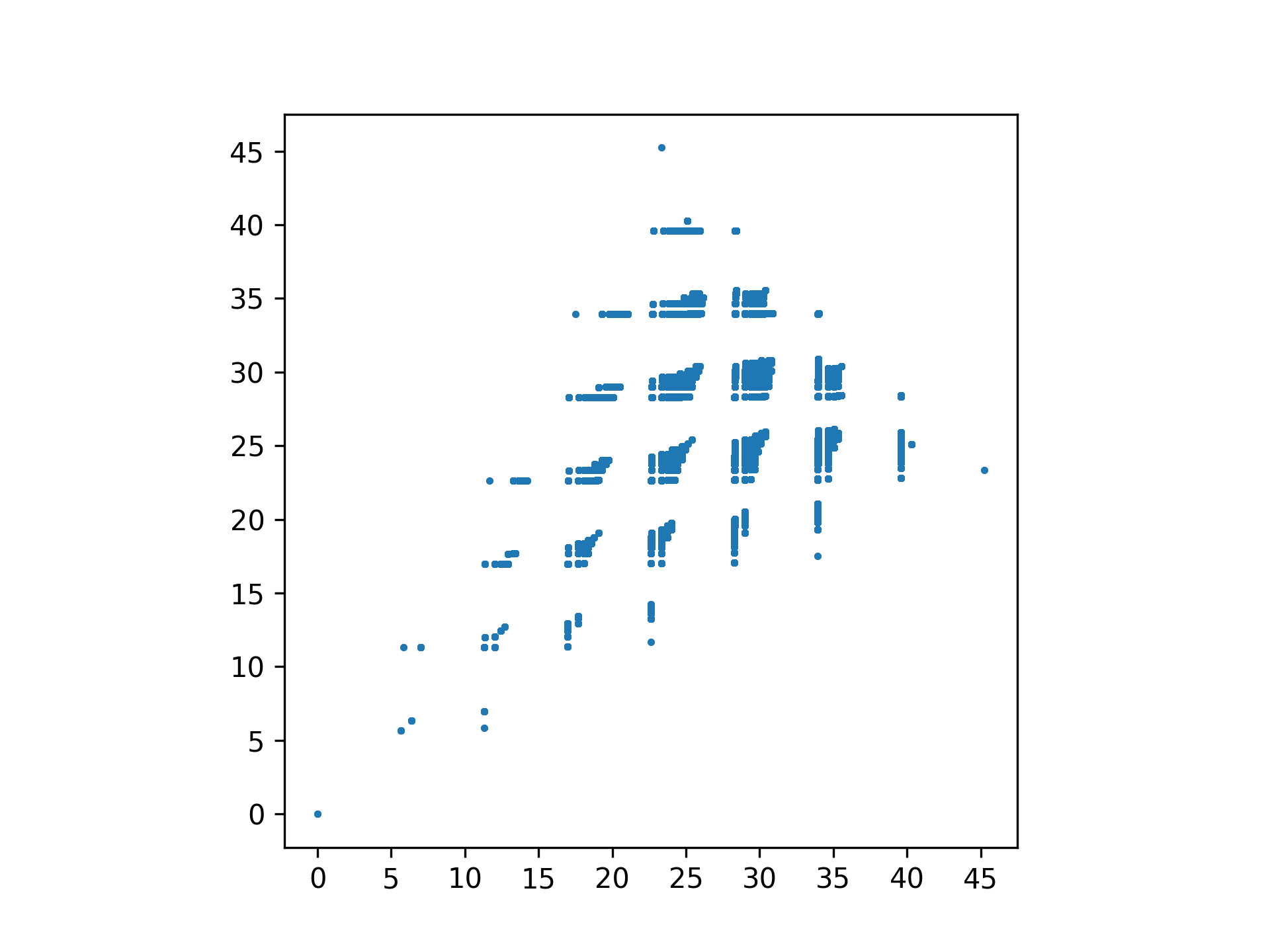}
\includegraphics[width=8cm]{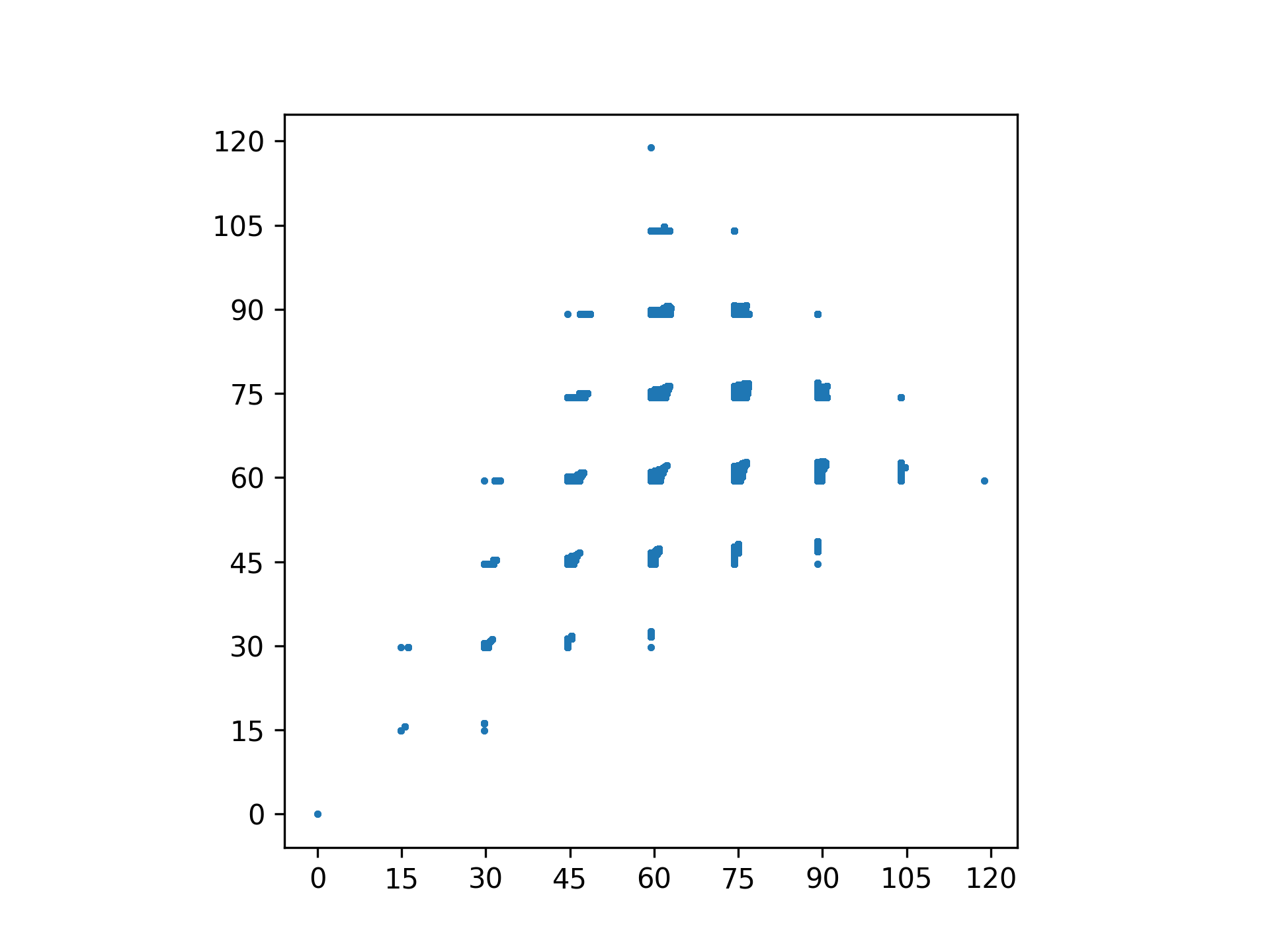}

We see that the Jordan projections increasingly tend to lie on an integral lattice. Theorem \ref{evlimit} gives a way to compute which lattice point each group element converges to. 

From an algebraic perspective, it not so surprising that the Jordan projections collect onto a lattice. Note that when $\lambda_1(g)$ is big, $\log|\lambda_1(g)|$ is approximately $\log|tr(g)|$, and $-\log|\lambda_3(g)|$ is approximately $\log|tr(g^{-1})|$. Trace functions $tr(\rho_t(\gamma))$ for $\gamma\in \Gamma_{pqr}$ are Laurent polynomials in $t$. This means that when $t$ is big, \[\phi(g) \approx \log(t)(d_1,d_2)\] where $d_1$ and $d_2$ are the highest powers of $t$ in $tr(\rho_t(\gamma))$ and $tr(\rho_t(\gamma^{-1}))$ respectively. From this perspective, theorem \ref{evlimit} tells us the ratios between the highest power of $t$ in $tr(\rho_t(\gamma))$.

The Labourie-Loftin parametrization can be modified to apply to our reflection orbifold $S$, by defining a holomorphic cubic differential on $S$ to be a holomorphic cubic differential $\tilde{\mu}$ on a universal cover which is preserved by orientation preserving deck transformations, and complex conjugated by orientation reversing deck transformations. The correspondence for quotient orbifolds (such as our triangle reflection orbifold) follows from the correspondence on a smooth, compact covering space, and the general fact that when solutions to PDE's are unique, they have to be invariant under the symmetry group of the input data. The correspondence between Higgs bundles and Hitchin components for orbifolds was worked out in generality in \cite{alessandrini_hitchin_2020}.  Note that this definition of cubic differential for orbifolds forces the fixed loci of reflections to be real trajectories of $\tilde{\mu}$. 

As predicted by the Labourie-Loftin correspondence, the space of cubic differentials on $S$ is $1$ dimensional. We now construct a nonzero element of this space. Consider a euclidian equilateral triangle $T\subset \CC$, whose sides are unit length, and such that the restriction of $dz^3$ to each side is real. Let $\bar{T}$ denote $T$ with the conjugate cubic differential $d\bar{z}^3$. We can glue copies of $T$ and $\bar{T}$ appropriately to get a Riemann surface with cubic differential which has an action of $\Gamma_{pqr}$. Call this cubic differential $\mu$. 

If we color copies of $T$ grey, and color copies of $\bar{T}$ white, and conformally map the resulting surface to the disk, we get the standard picture of the $pqr$ triangulation of the hyperbolic plane. Below is the picture for $p=q=r=4$.

\begin{center}
\includegraphics[width=7cm]{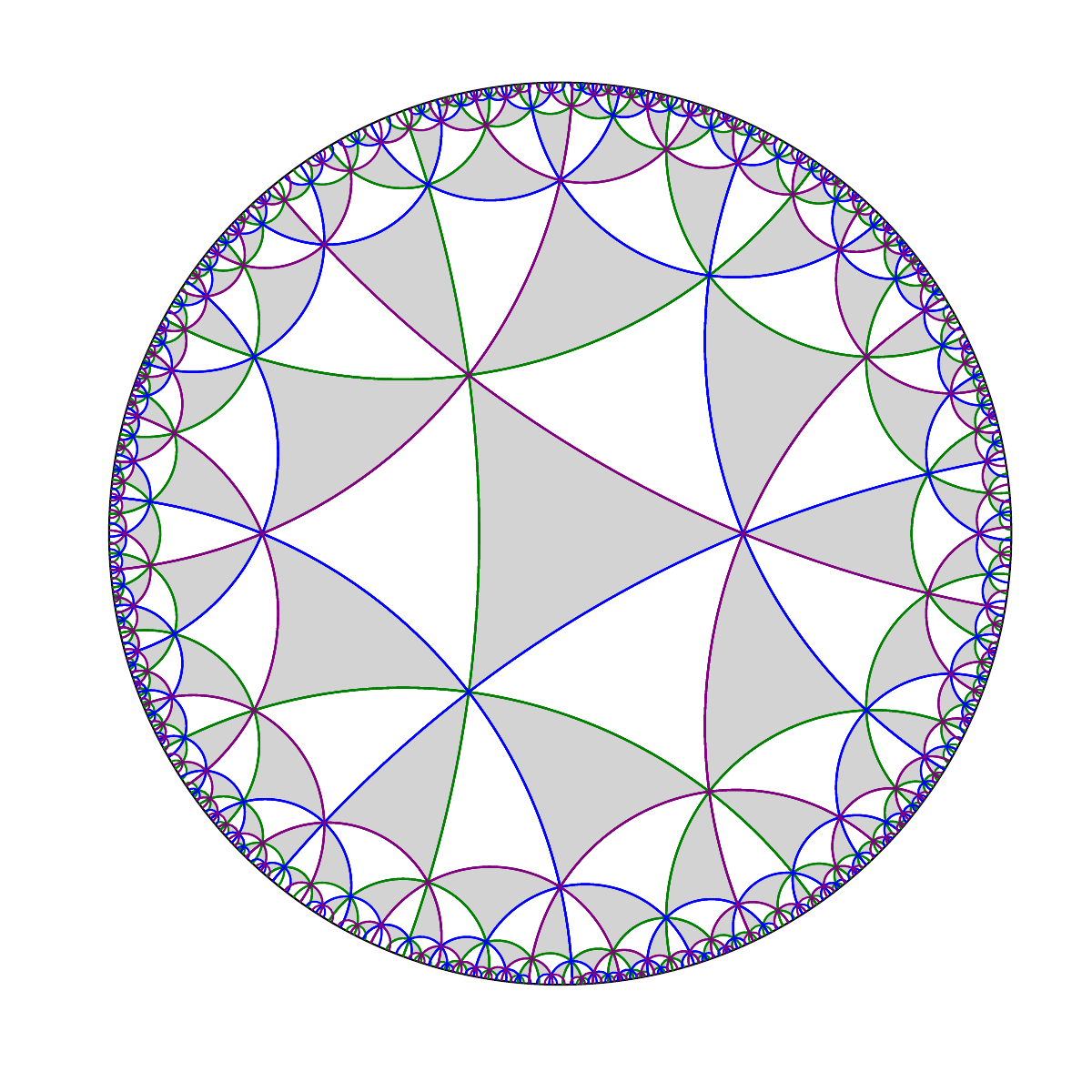}
\end{center}

We now illustrate how to apply theorem \ref{evlimit} with an example. The word 
\[w=cbcacbcacbcacbacbabcabab\]
was chosen arbitrarily from the words forming the cluster around the lattice point $(6,5)$ in the $t=10^{12}$ Jordan projection picture. We can compute the $F_\mu^\Delta$ translation lengths of $w$, and $w^{-1}$ as follows. First we construct a curve in the universal cover $\tilde{S}$ which is preserved by $w$ by concatinating straight line segemnts connecting centers of triangles, which cross reflection loci prescribed by the letters of $w$. Then pull this curve curve tight in the metric $h_\mu$. Then push it onto the reflection locus, while making sure not to change its $F_\mu^\Delta$ length. $F^\Delta_\mu$ assigns length $2$ to edges going counter clockwise around grey triangles, and length $1$ to edges going clockwise around grey triangles (and vice versa for white triangles) so we can just add up these numbers to compute translation length once we have a geodesic representative in the reflection locus.

\begin{center}
\includegraphics[trim={0 1cm 0 2cm},clip,width=13cm]{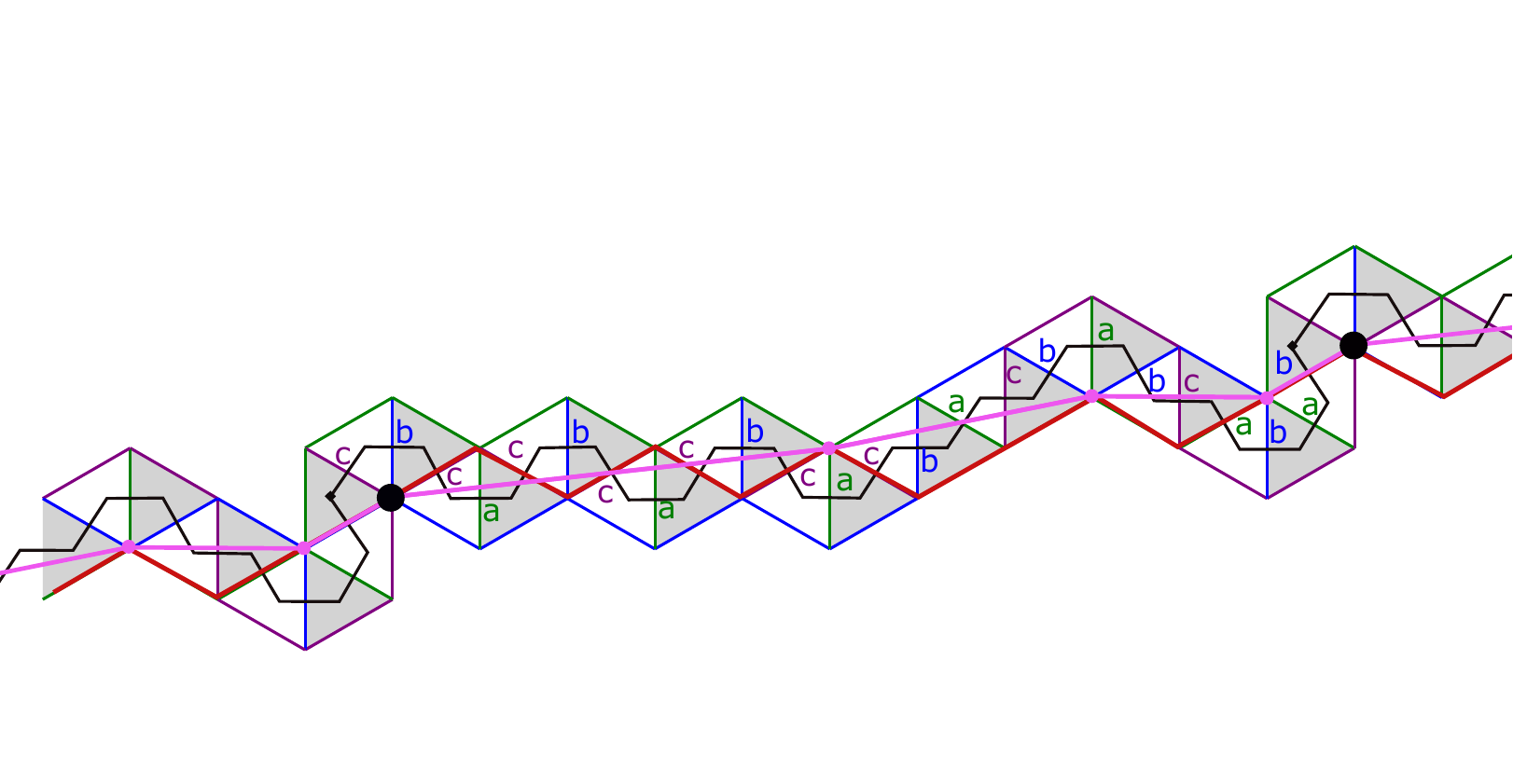}
\end{center}

Here, the thin black line represents the easy to construct $w$ invariant curve, the pink line represents the geodesic for the singular flat metric $h_\mu$, which is also a geodesic for $F_\mu^\Delta$, and the red line represents a geodesic for $F_\mu^\Delta$ which lies in the reflection locus. We see that the translation length of $w$ is $18$, and the translation length $w^{-1}$ is $15$. Theorem \ref{evlimit} thus predicts
\[\lim_{t\to\infty} \frac{\log(\lambda_1(\rho_t(w)))}{\log(\lambda_1(\rho_t(w^{-1})))} = \frac{6}{5}\]
which is what we observed in the Jordan projection picture. We can verify this limit rigorously by using computer algebra software to directly compute that the highest powers of $t$ in $tr(\rho_t(w))$ and $tr(\rho_t(w^{-1}))$ are $6$, and $5$ respectively. The reader can easily try these computations for other elements of triangle reflection groups.

\bibliographystyle{plain}
\bibliography{bib.bib}

\end{document}